\documentclass{amsart}

\usepackage{etex}
\usepackage{aliascnt}
\usepackage{bbm}
\usepackage{pbox}
\usepackage{booktabs}
\usepackage{arydshln}
\usepackage{pinlabel, subcaption}
\usepackage{esvect}
\usepackage{cals}
\newcommand{\bk}{\mathbbm{k}}
\newcommand{\KK}{K}

\newcommand{\OO}{{\mathcal O}}
\newcommand{\OK}{{\mathcal O}_\KK}
\newcommand{\m}{\to}

\newcommand{\cO}{\mathcal{O}}

\newcommand{\cT}{\mathcal{T}}

\providecommand{\Linkh}{\widehat{\ensuremath\mr{Link}}}
\providecommand{\Link}{\ensuremath\mr{Link}}

\newcommand{\bQ}{\mathbb{Q}}

\newcommand{\bZ}{\mathbb{Z}}
\newcommand{\Z}{\mathbb{Z}}
\newcommand{\Q}{\mathbb{Q}}
\newcommand{\C}{\mathbb{C}}
\newcommand{\R}{\mathbb{R}}

\usepackage{lmodern}
\usepackage{booktabs}
\usepackage[dvipsnames,svgnames,x11names,hyperref]{xcolor}
\usepackage{mathtools,url,graphicx,verbatim,amssymb,enumerate,stmaryrd,nicefrac}
\usepackage[pagebackref,colorlinks,citecolor=blue,linkcolor=blue,urlcolor=blue,filecolor=blue]{hyperref}
\usepackage{microtype}
\usepackage[margin=1.25in]{geometry}
\usepackage{pdflscape}

\usepackage{tikz, tikz-cd}
\usetikzlibrary{matrix,arrows,decorations.pathmorphing}
\usetikzlibrary{arrows,decorations.pathmorphing,backgrounds,fit,positioning,shapes.symbols,chains,calc}

\numberwithin{theoremcounter}{section}
\newaliascnt{theoremauto}{theoremcounter}

\newaliascnt{Defauto}{theoremcounter}

\newaliascnt{exampleauto}{theoremcounter}

\newaliascnt{lemmaauto}{theoremcounter}

\newaliascnt{propositionauto}{theoremcounter}

\newaliascnt{corollaryauto}{theoremcounter}

\newaliascnt{remarkauto}{theoremcounter}

\newaliascnt{notationauto}{theoremcounter}

\newaliascnt{claimauto}{theoremcounter}

\newaliascnt{warningauto}{theoremcounter}

\newaliascnt{questionauto}{theoremcounter}

\newaliascnt{discussionauto}{theoremcounter}

\newaliascnt{computationauto}{theoremcounter}

\newaliascnt{conjectureauto}{theoremcounter}

\newaliascnt{convauto}{theoremcounter}

\newtheorem{theorem}[theoremauto]{Theorem}
\newtheorem{lemma}[lemmaauto]{Lemma}
\newtheorem{proposition}[propositionauto]{Proposition}
\newtheorem{corollary}[corollaryauto]{Corollary}
\newtheorem*{corollary*}{Corollary}
\newtheorem{conjecture}[conjectureauto]{Conjecture}
\newtheorem{atheorem}{Theorem}

\theoremstyle{definition}
\newtheorem{definition}[Defauto]{Definition}
\newtheorem{notation}[notationauto]{Notation}
\newtheorem{convention}[convauto]{Convention}
\theoremstyle{remark}

\newtheorem{remark}[remarkauto]{Remark}

\newtheorem{question}[questionauto]{Question}

\newcommand{\cat}[1]{\mathsf{#1}}
\newcommand{\mr}[1]{{\rm #1}}

\newcommand{\lra}{\longrightarrow}

\DeclareMathOperator*{\sgn}{sgn}

\newcommand{\GL}{\mr{GL}}
\newcommand{\SL}{\mr{SL}}
\newcommand{\St}{\mr{St}}

\title[On the generalized Bykovski\u\i\, presentation]{On the generalized Bykovski\u\i\, presentation of Steinberg modules}

 \author{Alexander Kupers}
 \email{kupers@math.harvard.edu}
 \address{Harvard University \\ 
 Department of Mathematics \\
 	 1 Oxford Street \\
 	 Cambridge MA, 02138 \\USA}
  
\thanks{Alexander Kupers was supported in part by NSF grant DMS-1803766}

 \author{Jeremy Miller}\thanks{Jeremy Miller was supported in part by NSF grant DMS-1709726}
  \email{jeremykmiller@purdue.edu}
\address{Purdue University \\
Department of Mathematics \\
 	 150 North University \\
 	 West Lafayette IN, 47907 \\USA}
  
\author{Peter Patzt}
\email{patzt@math.ku.dk}
\address{University of Copenhagen\\
Department of Mathematic Sciences \\
 	 Universitetsparken 5 \\
 	 Copenhagen \O, DK-2100 \\Denmark}
	 \thanks{Peter Patzt was supported by the Danish National Research Foundation through the Copenhagen Centre for Geometry and Topology (DNRF151)}
   
\author{Jennifer C. H. Wilson}
\email{jchw@umich.edu}
\address{University of Michigan \\ Department of Mathematics \\
 	 530 Church St\\
 	 Ann Arbor MI, 48109 \\USA}
\thanks{Jennifer Wilson was supported in part by NSF grant DMS-1906123}

\AtBeginDocument{%
	\def\MR#1{}}

\date{\today}

\begin{document}
	
\begin{abstract} We study presentations of the virtual dualizing modules of special linear groups of number rings, the Steinberg modules. Bykovski\u\i\ gave a presentation for the Steinberg modules of the integers, and our main result is a generalization of this to the Gaussian integers and the Eisenstein integers. We also show that this generalization does not give a presentation for the Steinberg modules of several Euclidean number rings. \end{abstract}

\maketitle


\section{Introduction}


\subsection{Cohomology} In this paper, we study the cohomology of special linear groups of number rings in large degrees. Let $\OK$ denote the ring of integers in a number field $\KK$ with $r_1$ real embeddings and $r_2$ pairs of complex conjugate embeddings. It follows from the work of Borel--Serre \cite[Theorem 11.4.2]{BoSe} (also see e.g.~Church--Farb--Putman \cite[Section 1.4]{CFP}) that \[\nu_n \coloneqq \frac{r_1}{2}((n+1)n-2)+r_2 (n^2-1)-n+1\] is the virtual cohomological dimension of $\SL_n(\OK)$, and hence $H^i(\SL_n(\OK);\Q) = 0$ for $i>\nu_n$. This does not mean that $H^{\nu_n}(\SL_n(\OK);\Q) \neq 0$, only that there is some twisted coefficient system where this group is nontrivial. We investigate the following question.

\begin{question}
For $\OK$ a number ring, what is the largest $i$ such that $H^i(\SL_n(\OK);\Q)$ is non-zero? 
\end{question}

In particular, we seek better bounds on vanishing of rational cohomology than just the virtual cohomological dimension. See \cite{LS,LSK3,LSK45,CFPconj,CFP,CP,DSGGHSYK4,MPWY,DSEVKM} 
 for progress on this question as well as applications of this question to computations in algebraic $K$-theory. The main such results are Lee--Szczarba's theorem \cite[Theorem 1.3]{LS} that $H^{\nu_n}(\SL_n(\OK);\Q) = 0$ for $n \geq 2$ and $\OK$ a Euclidean domain, and Church--Putman's theorem \cite[Theorem A]{CP} that $H^{\nu_n-1}(\SL_n(\Z);\Q) = 0$ for $n \geq 3$. Our main theorem extends Church--Putman's result to two other number rings: the Gaussian integers $\bZ[i]$ and the Eisenstein integers $\bZ[\rho]$ with $\rho=\frac{1+\sqrt{-3}}{2}$ a sixth root of unity. These are Euclidean domains, with $\nu_n=n^2-n$. 

\begin{atheorem} \label{Vanishing}
Let $\OK$ denote the  Gaussian integers or Eisenstein integers. Then
\begin{align*}
H^{\nu_n-1}(\GL_n(\OK);\Q) &= 0 \qquad \text{ for $n \geq 2$}, \\
H^{\nu_n-1}(\SL_n(\OK);\Q) &= 0 \qquad \text{ for $n \geq 3$}. 
\end{align*} 
\end{atheorem}

\noindent In fact, it suffices to only invert $(2n+1)!$, but we restrict to rational statements in the introduction.

\subsection{Dualizing modules} 
\label{SecDualizing}

Our strategy for proving \autoref{Vanishing} is to give a presentation for the virtual dualizing module of the groups $\GL_n(\OK)$ and $\SL_n(\OK)$, and use it to show that the $(\nu_n-1)$st cohomology group vanishes. 

Recall that the \emph{Tits building} $\cT_n(\KK)$ of a field $\KK$ is the geometric realization of the poset of non-empty proper subspaces of the $\KK$-vector space $\KK^n$ ordered by inclusion. This poset is spherical of dimension $(n-2)$ \cite{Solomon,Garland,Quillen-Ki} and its top reduced homology is called the \emph{Steinberg module} and denoted $\St_n(\KK)$. The action of $\GL_n(\OK)$ on $\cT_n(\KK)$ gives this the structure of a $\bZ[\GL_n(\OK)]$-module. Borel--Serre \cite[\S 11]{BoSe} proved that for $\OK$ the ring of integers in a number field $\KK$, $\St_n(\KK)$ is the virtual dualizing module of $\KK$. That is, there is a natural isomorphism
\[H^{\nu_n-i}(\SL_n(\OK);\Q) \overset{\cong}\lra H_{i}(\SL_n(\OK);\St_n(\KK) \otimes \Q).\]
Thus, to show $H^{\nu_n-1}(\SL_n(\OK);\Q) =0$, it suffices to show $H_{1}(\SL_n(\OK);\St_n(\KK) \otimes \Q)=0$. 

This will be done by finding a presentation of the relevant Steinberg modules. For $\OO$ an integral domain, let $\mr{Byk}_n(\OO)$ denote the quotient of the free abelian group on symbols $[[\vec v_1,\ldots, \vec v_n]]$, with $\vec v_1,\ldots, \vec v_n$ an ordered basis of $\OO^n$, by the following relations:
\begin{enumerate}[\noindent (1)]
\item $[[\vec v_1,\ldots, \vec v_n]]=\sgn(\sigma)[[\vec v_{\sigma(1)},\ldots, \vec v_{\sigma(n)}]]$ for $\sigma$ a permutation of $\{1,\ldots,n\}$ and $\sgn(\sigma)$ its sign. 
\item $[[\vec v_1, \vec v_2, \ldots, \vec v_n]]=[[u \vec v_1, \vec v_2, \ldots, \vec v_n]]$ for $u$ a unit in $\OO$.
\item $[[\vec v_1,\vec v_2, \vec v_3, \ldots, \vec v_n]]-[[\vec v_1+\vec v_2,\vec v_2, \vec v_3, \ldots, \vec v_n]]+[[\vec v_1+\vec v_2,\vec v_1, \vec v_3, \ldots, \vec v_n]]=0$.
\end{enumerate} The symbols $[[\vec v_1,\ldots, \vec v_n]]$ are sometimes called \emph{modular symbols}. Letting $A \in \GL_n(\OO)$ act on a symbol $[[\vec v_1,\ldots, \vec v_n]]$ by $[[A\vec v_1,\ldots,A\vec v_n]]$ gives $\mr{Byk}_n(\OO)$ a $\bZ[\GL_n(\OO)]$-module structure.

Let $K$ denote the field of fractions of $\OO$. Given an ordered basis $\vec v_1,\ldots, \vec v_n$, the subposet of subspaces of $\KK^n$ which are spanned by a non-empty proper subset of $\vec v_1,\ldots, \vec v_n$, is isomorphic to the barycentric subdivision of $\partial \Delta^{n-1}$, and thus realizes to an $(n-2)$-dimensional sphere with canonical orientation. These spheres are called \emph{apartments} and sending $[[\vec v_1,\ldots, \vec v_n]]$ to the fundamental class of the apartment gives a homomorphism of $\bZ[\GL_n(\OO)]$-modules
\[\mr{Byk}_n(\OO) \lra \St_n(\KK).\] Bykovski\u\i\, \cite{Byk} proved that $\mr{Byk}_n(\Z) \m \St_n(\Q)$ is an isomorphism. If $\mr{Byk}_n(\OO) \m \St_n(\KK)$ is an isomorphism, we say \emph{the generalized Bykovski\u\i\, presentation holds for $\OO$}. That the Bykovski\u\i's  presentation holds for $\Z$ is the key ingredient in Church--Putman's vanishing result for $H^{\nu_n-1}(\SL_n(\Z);\Q)$. We investigate the following question, and give a partial answer:

\begin{question}For which number rings does the generalized Bykovski\u\i\ presentation hold?\end{question}

\begin{atheorem} \label{BforGandE} Let $\OK$ denote the Gaussian integers or Eisenstein integers, and $\KK$ its field of fractions. Then $\mr{Byk}_n(\OK) \m \St_n(\KK)$ is an isomorphism of $\bZ[\GL_n(\OK)]$-modules for all $n$.
\end{atheorem}

\autoref{Vanishing} follows quickly from \autoref{BforGandE}. Surjectivity of the map $\mr{Byk}_n(\OO) \m \St_n(\KK)$ follows from work of Ash--Rudolph \cite[Theorem 4.1]{AshRudolph} whenever $\OO$ is Euclidean. In fact, for number rings $\OK$, the generalized Riemann hypothesis implies that $\mr{Byk}_n(\OK) \m \St_n(\KK)$  is surjective if and only if $\OK$ is Euclidean \cite[Corollary 1.2]{MPWY}. One might think that the generalized Bykovski\u\i\, presentation holds for all Euclidean number rings, but this is not the case.

\begin{atheorem} \label{NoB}
Let $\OK$ be the ring of integers in $\KK=\Q(\sqrt d)$. Assume $\OK$ is a Euclidean domain that is not additively generated by units. Then the map $\mr{Byk}_n(\OK) \m \St_n(\KK)$ is not injective for all $n \geq 2$.
\end{atheorem}

The norm-Euclidean number rings satisfying the hypothesis of this theorem have been classified and are exactly $\bQ(\sqrt d)$ for $d\in \{-11,-7,-2,6,7,11,17,19,33,37,41,57,73\}$; see \autoref{Listd}. Thus there are many examples of Euclidean number rings where the generalized Bykovski\u\i\, presentation fails. In fact, we give a more general result (\autoref{EimpliesNoB}) allowing the reader to possibly find more examples.

\begin{remark} The main technical result is that certain simplicial complexes of ``augmented partial frames'' are highly-connected. This has applications in a forthcoming paper \cite{KMP} on homological stability for general linear groups of certain Euclidean domains.\end{remark}

\subsection{Acknowledgments} 

Much of this project was completed as part of the American Institute of Mathematics SQuaRE ``Secondary representation stability.'' We thank AIM for their support. We also thank Rohit Nagpal who participated in this SQuaRE but declined to be a coauthor. We thank our anonymous referee for their careful reading and thoughtful feedback. 

\tableofcontents
\setcounter{tocdepth}{2}

\section{Elementary properties of the Gaussian integers and Eisenstein integers}

In this section, we establish some elementary properties of the Gaussian integers and Eisenstein integers. These properties are the primary reason that the proof of \autoref{Vanishing} in this paper only applies to these rings.

\begin{notation} Let $\OK$ denote the Gaussian integers or Eisenstein integers. We will pick preferred ring generators for each ring, 
	\begin{align*}
	\mr{Gaussian\, integers}: &\qquad i, \\ 
	\mr{Eisenstein\, integers}: &\qquad \rho = e^{\frac{2\pi i}{6}} = \frac12 + i \frac{\sqrt3}{2}. 
	\end{align*} The latter is not the conventional choice of an additive generator for the Eisenstein integers, which is more typically $\rho^2 = e^{\frac{2\pi i}{3}}$. See \autoref{FigureUnits}.  
	\begin{figure}[h!]
\begin{subfigure}{.425\textwidth} \centering
\labellist
\Large \hair 0pt
\pinlabel { \color{violet} $\small{  \displaystyle 1 }$} [tr] at 75 54
\pinlabel { \color{violet} $\small{  \displaystyle i }$} [tr] at 54 75
\pinlabel { \color{red} $\small{  \displaystyle i^2=-1 }$} [tr] at 22 56
\pinlabel { \color{red} $\small{  \displaystyle i^3=-i }$} [tr] at 43 33
\endlabellist
\qquad
\includegraphics[scale=1.5]{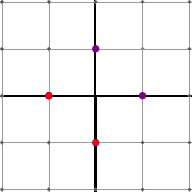}   \\[.2em]
\subcaption{The Gaussian integers.} 
\label{Gaussian-units}
\end{subfigure}\qquad  
\begin{subfigure}{.425\textwidth} \centering
\labellist
\Large \hair 0pt
\pinlabel { \color{violet} $\small{  \displaystyle 1 }$} [tr] at 88 53
\pinlabel { \color{violet} $\small{  \displaystyle \rho }$} [tr] at 78 73
\pinlabel { \color{red} $\small{  \displaystyle \rho^2=\rho-1 }$} [tr] at 42 76
\pinlabel { \color{red} $\small{  \displaystyle \rho^3=-1 }$} [tr] at 32 56
\pinlabel { \color{red} $\small{  \displaystyle \rho^4=-\rho }$} [tr] at 42 33
\pinlabel { \color{red} $\small{  \displaystyle \rho^5=1-\rho }$} [tr] at 110 33
\endlabellist
\qquad
\includegraphics[scale=1.5]{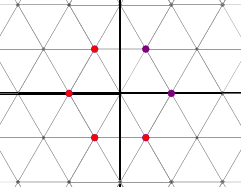}   \\[.2em]
\subcaption{The Eisenstein integers.} 
\label{Eisenstein-units}
\end{subfigure}  
\caption{The additive generators and the units in $\OK$.} 
\label{FigureUnits}
\end{figure}
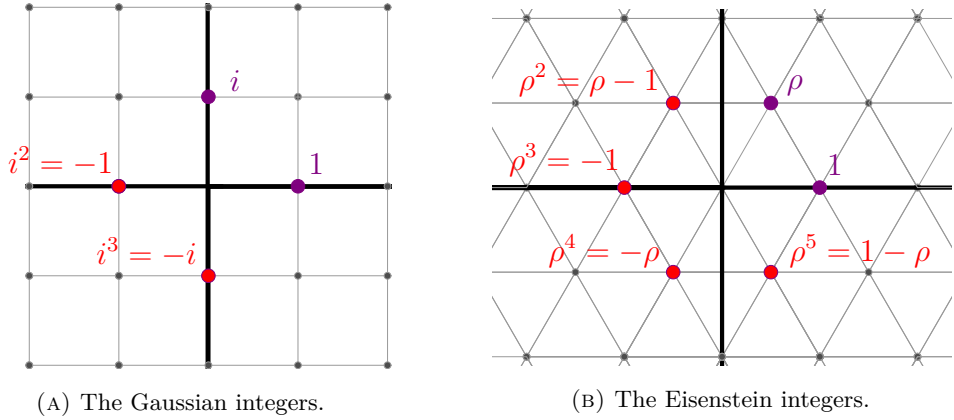
\noindent With this notation, observe that the complex norm is given as follows, 
\begin{align*}
\mr{Gaussian\, integers}: &&  |x+ iy|^2 &= x^2 + y^2 && (x,y \in \R) \\ 
\mr{Eisenstein\, integers}: && |x+ \rho y|^2 = \left|x+  \frac12 (1 + i \sqrt{3})y\right|^2 &= x^2 +xy+ y^2 && (x,y \in \R) 
\end{align*}
\end{notation}

The complex norm is a Euclidean function for $\OK$ the Gaussian integers or Eisenstein integers. That is, given $a,b \in \OK$ with $|b|>0$, then there is a $q \in \OK$ with $|a-qb|<|b|$. Moreover, this function is multiplicative in the sense that $|a b|=|a||b|$. 
We now prove some key lemmas. 

\begin{lemma} \label{lem0}
Let $\OK$ be the Gaussian integers or Eisenstein integers. Let $a,b \in \OK$ with $|a|=|b|>0$. Then there is a unit $u$ with $|a-ub|<|a|$. 
\end{lemma}

\begin{proof} 
Pick $u$ such that the angle between $a$ and $ub$ is less than $\pi/3$. It is an elementary exercise in trigonometry to see that $|a-ub|<|a|$.
\end{proof}

\begin{convention}
Let $H$ be a group and $S$ a set of generators. In this paper, the term \emph{Cayley graph} for $H$ with generators $S$ means the undirected graph with vertices elements of $H$ and an edge between $h$ and $g$ if and only if they differ by left multiplication by an element in $S$. In particular, if $s$ and $s^{-1} \in S$, then the Cayley graph for $H$ with generators $S$ agrees with the Cayley graph for $H$ with generators $S \setminus 
\{s^{-1}\}$.
\end{convention}

\begin{lemma} \label{lem1}
Let $\OK$ be the Gaussian integers or Eisenstein integers. Let $G$ be the Cayley graph of $\OK$ with units as generators.  Let $z \in \C$ and let $G_z$ be the full subgraph of $G$ on vertices $x$ with $|x-z|<1$. Then $G_z$ is connected. 
\end{lemma}

\begin{proof}  
Consider the open ball $B_z$ of radius $1$ centered on the point $z \in \C$. This ball must contain at least one element of $\OK$, and without loss of generality we may assume it contains $0$.

We will first consider the case where $\OK$ is the Eisenstein integers. Any other element of $\OK$ in $B_z$ must have distance $<2$ from the origin. There are only twelve such points, as shown in \autoref{Eisenstein-norm2}. Six of these (colored dark gray) are joined to $0$ by an edge, and the other six (colored light gray) are distance $2$ from the origin in the edge metric on the Cayley graph. 

\begin{figure}[h!]
\begin{subfigure}[t]{.47\textwidth} \centering
\labellist
\Large \hair 0pt
\pinlabel { \color{darkgray} $\tiny{  \displaystyle \rho}$} [tr] at 71 79
\pinlabel { \color{darkgray} $\tiny{  \displaystyle 1 }$} [tr] at 82 59
\pinlabel { \color{black} $\tiny{  \displaystyle 0 }$} [tr] at 58 59
\pinlabel { \color{gray} $\tiny{  \displaystyle \rho+1}$} [tr] at 112 81
\endlabellist
\qquad
\includegraphics[scale=1.2]{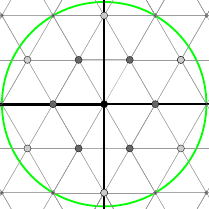}   \\[.2em]
\subcaption{Eisenstein integers within complex distance 2 of the origin.} 
\label{Eisenstein-norm2}
\end{subfigure}
\qquad
\begin{subfigure}[t]{.47\textwidth} \centering
\labellist
\Large \hair 0pt
\endlabellist
\qquad
\includegraphics[scale=1.2]{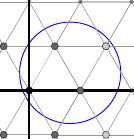}   \\[.2em]
\subcaption{An instance of a ball $B_z$ containing $0$ and $\rho+1$.} 
\label{InscribedAngles}
\end{subfigure}
\caption{}
\end{figure}
Up to symmetry, then, it suffices to assume that both $0$ and $\rho + 1$ are contained in $B_z$, and show that either $\rho$ or $1$ must be contained in $B_z$. See \autoref{InscribedAngles}.  In this case, $z$ must be contained in the intersection $B_0 \cap B_{\rho + 1}$, as in \autoref{EisensteinCircleIntersection}. But this intersection is contained in the union $B_{\rho} \cup B_1$, as in \autoref{EisensteinCircleUnion}. Hence $\rho \in B_z$ or $1 \in B_z$. 

\begin{figure}[h!]
\begin{subfigure}[t]{.45\textwidth} \centering
\labellist
\Large \hair 0pt
\pinlabel { \color{black} $\tiny{  \displaystyle 0 }$} [tr] at 40 43
\pinlabel { \color{gray} $\tiny{  \displaystyle \rho+1}$} [tr] at 95 66
\pinlabel { \color{darkgray} $\tiny{  \displaystyle \rho}$} [tr] at 53 64
\pinlabel { \color{darkgray} $\tiny{  \displaystyle 1 }$} [tr] at 65 43
\endlabellist
\qquad
\includegraphics[scale=1.2]{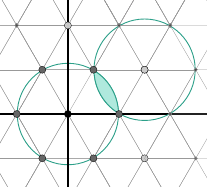}   \\[.2em]
\subcaption{The intersection $B_0 \cap B_{\rho + 1}$.} 
\label{EisensteinCircleIntersection}
\end{subfigure}
\qquad
\begin{subfigure}[t]{.45\textwidth} \centering
\labellist
\Large \hair 0pt
\endlabellist
\qquad
\includegraphics[scale=1.2]{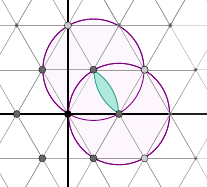}   \\[.2em]
\subcaption{The containment $B_0 \cap B_{\rho + 1} \subseteq B_{\rho} \cup B_{1} $.} 
\label{EisensteinCircleUnion}
\end{subfigure}
\caption{}
\end{figure}
%
%


Next suppose $\OK$ is the Gaussian integers. There are only eight points other than 0 that could be contained in the ball $B_z$, as in \autoref{Gaussian-norm2}. 
\begin{figure}[h!]
\labellist
\Large \hair 0pt
\pinlabel { \color{darkgray} $\tiny{  \displaystyle i}$} [tr] at 54 78
\pinlabel { \color{darkgray} $\tiny{  \displaystyle 1 }$} [tr] at 78 55
\pinlabel { \color{black} $\tiny{  \displaystyle 0 }$} [tr] at 54 55
\pinlabel { \color{gray} $\tiny{  \displaystyle i+1}$} [tr] at 90 78
\endlabellist
\qquad
\includegraphics[scale=1.3]{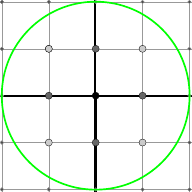}   \\[.2em]
\caption{Gaussian integers within complex distance 2 of the origin} 
\label{Gaussian-norm2}
\end{figure}
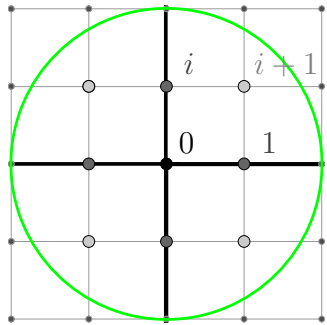
It suffices to check that, if $0$ and $i+1$ are contained in $B_z$, then so is one of $i$ or $1$. But $B_0 \cap B_{i + 1} \subseteq B_{i} \cup B_{1} $, as shown in \autoref{GaussianCircle}.

\begin{figure}[h!]
\begin{subfigure}[t]{.45\textwidth} \centering
\labellist
\Large \hair 0pt
\pinlabel { \color{black} $\tiny{  \displaystyle 0 }$} [tr] at 53 43
\pinlabel { \color{gray} $\tiny{  \displaystyle i+1}$} [tr] at 90 76
\pinlabel { \color{darkgray} $\tiny{  \displaystyle i}$} [tr] at 53 76
\pinlabel { \color{darkgray} $\tiny{  \displaystyle 1 }$} [tr] at 75 43
\endlabellist
\qquad
\includegraphics[scale=1.3]{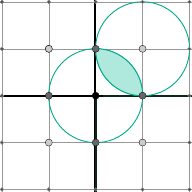}   \\[.2em]
\subcaption{The intersection $B_0 \cap B_{i + 1}$.} 
\label{EisensteinCircleIntersection}
\end{subfigure}
\qquad
\begin{subfigure}[t]{.45\textwidth} \centering
\labellist
\Large \hair 0pt
\endlabellist
\qquad
\includegraphics[scale=1.3]{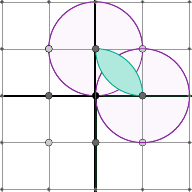}   \\[.2em]
\subcaption{The containment $B_0 \cap B_{i + 1} \subseteq B_{i} \cup B_{1} $.} 
\label{EisensteinCircleUnion}
\end{subfigure}
\caption{} \label{GaussianCircle}
\end{figure}

%
%
%
The result follows. 
\end{proof} 

\begin{lemma} \label{lem2G} Let $\OK$ be the Gaussian integers.  If $z_1$ and $z_2$ are any complex numbers, then there exist $r_1, r_2 \in \OK$ so that \[|z_1-r_1|<1, \quad |z_2-r_2|<1, \quad \text{and} \quad |(z_1-r_1)+(z_2-r_2)|<1.\]
\end{lemma}

\begin{proof} Up to the addition of elements of $\OK$, we may assume that $z_1, z_2$ are contained in the set
\[S = \left\{ z \in \C \; \middle| \; \frac{-1}{2} \leq \mr{Re}(z) \leq \frac 12, \;  \frac{-1}{2} \leq  \mr{Im}(z) \leq \frac 12  \right\},\]
the closure of the square fundamental domain for $\OK$ centered around zero  in \autoref{GaussianNorm1Tiles}. 
%
%

%
%
\begin{figure}[h!]
\labellist
\Large \hair 0pt
\pinlabel { \color{black} $\small{  \displaystyle S }$} [tr] at 58 48
\pinlabel { \color{violet} $\small{  \displaystyle \rho=i }$} [tr] at 70 71
\pinlabel { \color{violet} $\small{  \displaystyle 1 }$} [tr] at 77 48
\endlabellist
\qquad
\includegraphics[scale=1.5]{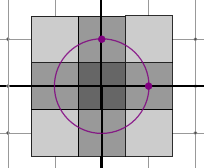}   \\[.2em]
\caption{The region $S$ and its translates $S \pm i$, $S\pm 1$, $S\pm i \pm 1$.} 
\label{GaussianNorm1Tiles}
\end{figure}

The sum $z_1+z_2$ must have both real and imaginary parts in the interval $[-1,1]$, so the sum is contained in $S$ or in one of its eight translates shown in \autoref{GaussianNorm1Tiles}. 
If $|z_1+z_2| \leq 1$, then we are done, so suppose otherwise.  Up to symmetry, we may consider two cases: $(z_1+z_2) \in (S+1)$ or $(z_1+z_2) \in (S+1+i)$. 

First suppose that $(z_1+z_2) \in (S+1)$. This means that at least one of $z_1$ and $z_2$ (say, $z_1$) must have real part at least $\frac 14$. Then we will replace $z_1$ by $(z_1-1)$. By assumption 
\[\frac {-3}{4} \leq \mr{Re}(z_1 - 1) \leq \frac{-1}{2} \qquad \frac {-1}{2} \leq \mr{Im}(z_1 - 1)  \leq \frac{1}{2}\]
and so $(z_1-1)$ lies in the rectangular region shown in \autoref{GaussianNorm1Tiles-quarter}. The number $(z_1 - 1)$ is contained in the unit ball, as
\[\left(  \frac {-3}{4} \right)^2 +  \left( \frac{1}{2} \right)^2 <1;\] see \autoref{GaussianNorm1Tiles-quarter}. But now $( (z_1 - 1) + z_2)$ is contained in $S$ and therefore in the unit ball, so we have completed this case. 
\begin{figure}[h!]
\labellist
\Large \hair 0pt
\endlabellist
\qquad
\includegraphics[scale=1.5]{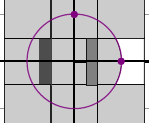}   \\[.2em]
\caption{The regions containing $z_1$ and $(z_1-1)$ are shaded gray and dark gray, respectively. The region containing $(z_1+z_2)$ is white.} 
\label{GaussianNorm1Tiles-quarter}
\end{figure}

Now suppose that $(z_1+z_2) \in (S+1+i)$. Again we may assume that  $z_1$ has real part at least $\frac 14$, and again we know $|z_1 - 1|<1$.  If $|z_1 - 1 - i| <1$, then we could replace $z_1$ by $(z_1 - 1 - i)$, and the sum $(z_1 - 1 - i) + z_2$ would be contained in $S$. So suppose $|z_1 - 1 - i|  \geq 1$. See \autoref{GaussianNorm1Tiles-Case2b}. In this case we will replace $z_1$ by $(z_1 - 1)$ and $z_2$ by $(z_2 - i)$, and then 
\[(z_1 - 1) + (z_2 - i) \in S\]
is contained in the unit ball as desired. 
It remains to show that $|z_2 - i| <1$. 
\begin{figure}[h!]
\begin{subfigure}{.425\textwidth} \centering
\labellist
\Large \hair 0pt
\endlabellist
\qquad
\includegraphics[scale=1.5]{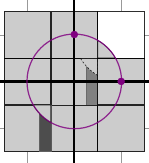}   \\[.2em]
\subcaption{The regions containing $z_1$ and $(z_1-1-i)$ are shaded gray and dark gray, respectively. The region containing $(z_1+z_2)$ is white.} 
\label{GaussianNorm1Tiles-Case2b}
\end{subfigure} 
\qquad 
\begin{subfigure}{.425\textwidth} \centering
\labellist
\Large \hair 0pt
\endlabellist
\qquad
\includegraphics[scale=1.5]{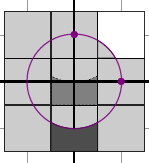}   \\[.2em]
\subcaption{The regions containing $z_2$ and $(z_2-i)$ are shaded gray and dark gray, respectively. The region containing $(z_1+z_2)$ is white.} 
\label{GaussianNorm1Tiles-Case2b2}
\end{subfigure}  
\caption{}
\end{figure}

Assume for contradiction that $|z_2 - i|  \geq 1$, so $z_2$ is contained in the region shown in \autoref{GaussianNorm1Tiles-Case2b2}. This assumption implies in particular that $$\mr{Im}(z_2) \leq 1- \frac{\sqrt{3}}{2}.$$  The assumption that $|z_1 - 1 - i|  \geq 1$ implies that \[\mr{Im}(z_1) \leq 1-\frac{\sqrt{7}}{4},\] as we see in \autoref{GaussianNorm1Tiles-Case2b}. But then 
\[\mr{Im}(z_1 + z_2) \leq \left(1- \frac{\sqrt{3}}{2}\right) +  \left(1- \frac{\sqrt{7}}{4}\right) < \frac 12,\]
which contradicts the premise that $z_1+ z_2$ is contained in the region $(S+1+i)$. So we conclude that $|z_2 - i| <1$, which concludes the proof. 
%
%
%
\end{proof} 

\begin{lemma} \label{lem2E} Let $\OK$ be the Eisenstein integers.  If $z_1$ and $z_2$ are any complex numbers, then there exist $r_1, r_2 \in \OK$ so that \[|z_1-r_1|<1, \quad |z_2-r_2|<1, \quad \text{and} \quad |(z_1-r_1)+(z_2-r_2)|<1.\]
\end{lemma}

\begin{proof} Consider the rectangular fundamental domains for $\OK$ shown in \autoref{EisensteinRectangles}. Up to the addition of elements of $\OK$, we may assume that $z_1$ and $z_2$ are in (the closure of) the fundamental domain centered about zero, 
\[S = \left\{ z \in \C \; \middle| \; \frac{-1}{2} \leq \mr{Re}(z) \leq \frac 12, \;  \frac{-\sqrt{3}}{4} \leq  \mr{Im}(z) \leq \frac{\sqrt{3}}{4} \right\}.\]
The region $S$ is shaded medium gray in \autoref{EisensteinRectangles}.  Observe that $|z_1|, |z_2| <1$. 

\begin{figure}[h!]
\begin{subfigure}{.4\textwidth} \centering
\labellist
\Large \hair 0pt
\pinlabel { \color{violet} $\small{  \displaystyle \rho }$} [tr] at 70 72
\pinlabel { \color{violet} $\small{  \displaystyle 1 }$} [tr] at 80 53
\endlabellist
\qquad
\includegraphics[scale=1.5]{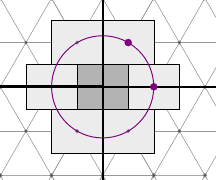}   \\[.2em]
\subcaption{The unit circle and a rectangular fundamental domain for $\OK$ along with six of its translates.} 
\label{EisensteinRectangles}
\end{subfigure} \qquad
\begin{subfigure}{.4\textwidth} \centering
\labellist
\Large \hair 0pt
\pinlabel { \color{violet} $\small{  \displaystyle \rho }$} [tr] at 70 72
\pinlabel { \color{violet} $\small{  \displaystyle 1 }$} [tr] at 80 53
\endlabellist
\qquad
\includegraphics[scale=1.5]{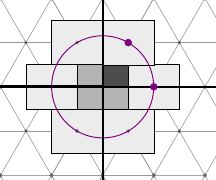}   \\[.2em]
\subcaption{We may assume $z_1$ is in the dark gray region.} 
\label{EisensteinRectangles-quadrant}
\end{subfigure}
\caption{}
\end{figure}

By our symmetry under reflection in the real and imaginary axes, we may assume that Im$(z_1)$, Re$(z_1) \geq 0$. Thus $z_1$ is in the quadrant shaded dark gray in \autoref{EisensteinRectangles-quadrant}, and 
$$ 
 \frac{-1}{2}  \leq \mathrm{Re}(z_1 + z_2) \leq 1 , \qquad   \qquad
 \frac{-\sqrt{3}}{4}   \leq  \mathrm{Im}(z_1 + z_2) \leq  \frac{\sqrt{3}}{2}.   $$ 
We deduce that the sum $(z_1 + z_2)$ must be contained in the rectangle shown in \autoref{EisensteinRectangles-sum}.

\begin{figure}[h!]
\begin{subfigure}{.4\textwidth} \centering
\labellist
\Large \hair 0pt
\pinlabel { \color{violet} $\small{  \displaystyle \rho }$} [br] at 58 50
\pinlabel { \color{violet} $\small{  \displaystyle 1 }$} [br] at 70 30
\pinlabel { \color{violet} $\small{  \displaystyle \rho -1 }$} [br] at 25 49
\endlabellist
\qquad
\includegraphics[scale=2]{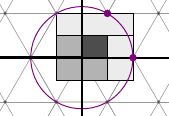}   \\[.2em]
\subcaption{The sum $z_1+z_2$ is contained in the rectangle} 
\label{EisensteinRectangles-sum}
\end{subfigure} \qquad
\begin{subfigure}{.4\textwidth} \centering
\labellist
\Large \hair 0pt
\pinlabel { \color{violet} $\small{  \displaystyle \rho }$} [br] at 58 50
\pinlabel { \color{violet} $\small{  \displaystyle 1 }$} [br] at 70 30
\pinlabel { \color{violet} $\small{  \displaystyle \rho -1 }$} [br] at 25 49
\pinlabel { \color{darkgray} $\small{  \displaystyle A }$} [br] at 70 42
\pinlabel { \color{gray} $\small{  \displaystyle B }$} [br] at 70 20
\endlabellist
\qquad
\includegraphics[scale=2]{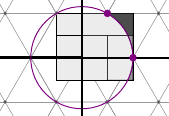}   \\[.2em]
\subcaption{The region $A$ is shown in dark gray and $B$ in medium gray} 
\label{EisensteinRectangles-LargeSum}
\end{subfigure} 
\caption{}
\end{figure}

If $|z_1+z_2|<1$ then we are done, so assume otherwise. In this case, either $z_1+z_2=\rho-1$, or $(z_1+z_2)$ is in one of the regions $A$ or $B$ shown in \autoref{EisensteinRectangles-LargeSum}.

If 
$z_1+z_2=\rho-1$, then we must have the extremal values of $z_1$ and $z_2$ given by $$ z_1 =  \left(\frac{\sqrt{3}}{4}\right)i \qquad \text{ and }  z_2 = -\frac12 + \left(\frac{\sqrt{3}}{4}\right)i, $$ as shown in \autoref{EisensteinRectangles-SumCorner}. In this case both $z_1$ and $z_1+z_2$ are strictly within distance $1$ of $\rho-1$. We may then replace $z_1$ by $(z_1 - \rho +1)$, so
$$ |z_1 - \rho +1| <1 \qquad \text{and} \qquad |(z_1 - \rho +1) + z_2| =0 <1.$$

Next suppose $z_1 + z_2 \in A$. Necessarily $z_1 \neq 0$, otherwise $z_1+z_2 = z_2$ would have norm strictly less than one. But then both $z_1$ and $z_1 + z_2$ are  within strict distance 1 of $\rho$, as in \autoref{EisensteinRectangles-SumA}. We may replace $z_1$ with $z_1 - \rho$, so
$$ |z_1 - \rho | <1 \qquad \text{and} \qquad |(z_1 - \rho) + z_2| <1.$$

\begin{figure}[h!]
\begin{subfigure}{.4\textwidth} \centering
\labellist
\Large \hair 0pt
\pinlabel { \color{violet} $\small{  \displaystyle \rho }$} [br] at 58 50
\pinlabel { \color{violet} $\small{  \displaystyle 1 }$} [br] at 70 30
\pinlabel { \color{violet} $\small{  \displaystyle \rho -1 }$} [br] at 25 49
\pinlabel { \color{red} $\small{  \displaystyle z_2}$} [br] at 26 32
\pinlabel { \color{red} $\small{  \displaystyle z_1}$} [br] at 38 32
\endlabellist
\qquad
\includegraphics[scale=2]{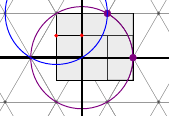}   \\[.2em]
\subcaption{The case $z_1+z_2 = \rho-1$} 
\label{EisensteinRectangles-SumCorner}
\end{subfigure} \qquad
\begin{subfigure}{.4\textwidth} \centering
\labellist
\Large \hair 0pt
\pinlabel { \color{violet} $\small{  \displaystyle \rho }$} [br] at 58 50
\pinlabel { \color{violet} $\small{  \displaystyle 1 }$} [br] at 70 30
\pinlabel { \color{violet} $\small{  \displaystyle \rho -1 }$} [br] at 25 49
\pinlabel { \color{darkgray} $\small{  \displaystyle A }$} [br] at 70 42
\endlabellist
\qquad
\includegraphics[scale=2]{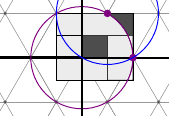}   \\[.2em]
\subcaption{The case $ z_1 + z_2 \in A$ } 
\label{EisensteinRectangles-SumA}
\end{subfigure} \qquad
\caption{}
\end{figure}

Finally, suppose that $z_1 + z_2 \in B$. It follows that  $$ \mathrm{Re}(z_1 + z_2) \geq \frac{ \sqrt{13}}{4} > \frac12. $$  This implies that Re$(z_i) \geq \frac 14 $ for at least one choice of $i=1, 2$, so $z_i$ is contained in the dark gray region in \autoref{EisensteinRectangles-SumB}. But then both $z_i$ and $z_1 +z_2$ are within distance strictly less than $1$ of $1$, as in \autoref{EisensteinRectangles-SumB}. We may replace $z_i$ with $z_i-1$, and
$$ |z_i - 1| <1 \qquad \text{and} \qquad |z_1 +  z_2 -1| <1.$$

\begin{figure}[h!] \centering
\labellist
\Large \hair 0pt
\pinlabel { \color{violet} $\small{  \displaystyle \rho }$} [br] at 58 50
\pinlabel { \color{violet} $\small{  \displaystyle 1 }$} [br] at 70 30
\pinlabel { \color{violet} $\small{  \displaystyle \rho -1 }$} [br] at 25 49
\pinlabel { \color{gray} $\small{  \displaystyle B }$} [br] at 70 20
\endlabellist
\qquad
\includegraphics[scale=2]{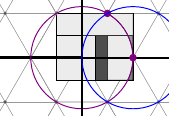}   \\[.2em]
\caption{ The case $ z_1 + z_2 \in B$. The region of  $S$ with real part at least $\frac14$ is shaded dark gray.} 
\label{EisensteinRectangles-SumB}
\end{figure}
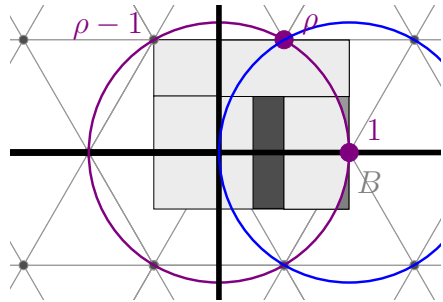 

This concludes the proof. 
\end{proof}

\section{Connectivity results}

Our results will be a consequence of connectivity/non-connectivity results for complexes of augmented partial frames. These complexes were introduced by Church--Putman \cite{CP} to give a topological proof of Bykovski\u\i's presentation of $\St_n(\Q)$. The original proof used more algebraic methods \cite{Byk}.

\subsection{Definitions and previously known results} We say that a simplicial complex is  \emph{$d$-spherical} if it is simultaneously $d$-dimensional and $(d-1)$-connected, in which case it is homotopy equivalent to a wedge of $d$-spheres. A simplicial complex is \emph{Cohen--Macaulay} of dimension $d$ if it is $d$-spherical and the link of every $k$-simplex is $(d-k-1)$-spherical. We follow the usual convention that $(-1)$-connected means non-empty. An example of a Cohen--Macaulay complex is the Tits building:

\begin{definition} For a finite-dimensional vector space $V$ over a field $\KK$, let $\cT(V)$ denote the geometric realization of the poset of proper non-empty subspaces of $V$ ordered by inclusion. When $V=\KK^n$, write $\cT_n(\KK)$ for $\cT(\KK^n)$ and call it the \emph{$n$th Tits building} of $\KK$.\end{definition} 

The following theorem seems to have first appeared in Solomon \cite{Solomon} in the case of finite fields; see Garland \cite[Theorem 2.2]{Garland} and Quillen \cite[Theorem 2]{Quillen-Ki} for the general case. 

\begin{theorem}[Solomon--Tits] \label{SolomonTits} For $\KK$ a field, $\cT_n(\KK)$ is Cohen--Macaulay of dimension $(n-2)$.\end{theorem}

Since $\cT_n(\KK)$ is $(n-2)$-spherical, 
\[\St_n(\KK) \coloneqq \widetilde H_{n-2}(\cT_n(\KK))\] is the only possible non-zero reduced homology group, called the $n$th \emph{Steinberg module}. It may be helpful to remark that for $\OO$ a Dedekind domain with field of fractions $\KK$, there is a natural bijection between summands of $\OO^n$ and subspaces of $\KK^n$; see e.g. Church--Farb--Putman \cite[Lemma 2.3]{CFP}.

%


\subsubsection{Partial frames}

 We now recall the definition of the complexes of (augmented) partial frames. The complex of partial frames is closely related to the complex of partial bases considered (for example) by Maazen \cite{Maazen} and van der Kallen \cite{vdKallen} 
in the context of homological stability. See Church--Farb--Putman \cite{CFP} and \cite{MPWY} for a discussion of the relationship between the complex of partial bases/frames and generators of Steinberg modules. From now on, we let $\OO$ denote an integral domain. 

\begin{convention}
In this paper, a \emph{line} in $\OO^n$ will mean a rank one free summand. 
\end{convention}

\begin{definition}
	A vector $\vec v \in \OO^n$ is called \emph{primitive} if its span is a direct summand. In that case, we denote its span by $v$. Similarly if $v$ is a line, we let $\vec v$ denote a primitive vector which spans that line. The vector $\vec v$ is well-defined up to multiplication by a unit of $\OO$.\end{definition}

\begin{definition}A \emph{partial frame} is an unordered collection of lines $ v_0, \ldots,  v_p$ in $\OO^n$ such that there are lines $ v_{p+1},\ldots,  v_{n-1}$ so that the natural map $v_0 \oplus \cdots \oplus  v_{n-1} \to \OO^n$ is an isomorphism.\end{definition}

\begin{definition} \label{complexesDef}
	Let $M$ be a finite-rank free $\OO$-module. The \emph{complex of partial frames} $B(M)$ is the simplicial complex with $p$-simplices given by the set of partial frames in $M$ of cardinality $p+1$. A simplex $\{w_0,\ldots,  w_q\}$ is a face of $\{v_0,\ldots,  v_p\}$ if and only if $\{ w_0,\ldots,  w_q \} \subseteq \{ v_0,\ldots,  v_p \}$. We write $B_n(\OO)$ for $B(\OO^n)$. \end{definition}

We consider $\OO^m$ as a submodule of $\OO^{m+n}$ by the inclusion of the first $m$ coordinates, and let $e_1,\ldots, e_m$ denote the lines spanned by the standard basis vectors of $\OO^m$. As in Church--Putman \cite[Definition 4.1]{CP}, we use the shorthand
\[B_n^m(\OO) \coloneqq \mr{Link}_{B_{n+m}(\OO)}(e_1,\ldots,  e_m).\]
Observe that $B_n^0(\OO)$ is equal to $B_n(\OO)$.

\begin{theorem} \label{thm:maazen}
	If $\OO$ is a Euclidean domain, then the simplicial complexes $B_n^m(\OO)$ are Cohen--Macaulay of dimension $(n-1)$.
\end{theorem} 

We prove this theorem using Maazen \cite{Maazen}, though it can also be established by adapting the techniques of Church--Putman \cite[Theorem 4.2]{CP}, who prove it for the case $\OO=\Z$.

\begin{proof}[Proof of \autoref{thm:maazen}] Since the link of a $p$-simplex in $B_n^m(\OO)$ is isomorphic to $B_{n-p-1}^{m+p+1}(\OO)$, it suffices to show that the complexes $B_n^m(\OO)$ are $(n-2)$-connected. Let $B_n^m(\OO)_\bullet$ denote the semi-simplicial set of \emph{ordered partial frames}. That is, $B_n^m(\OO)_p$ is the set of ordered $(p+1)$-tuples $( v_0,\ldots, v_p)$ with underlying set $\{ v_0,\ldots, v_p\}$ a $p$-simplex of $B_n^m(\OO)$. The $i$th face map forgets the $i$th line. By  \cite[Lemma 3.16]{KM2}, it suffices to show $B_n^m(\OO)_\bullet$ is $(n-2)$-connected. Let $\mr{O}_n^m(\OO)_\bullet$ denote the semi-simplicial set of \emph{ordered partial bases}. That is, $\mr{O}_n^m(\OO)_p$ is the set of ordered $(p+1)$-tuples $(\vec v_0,\ldots,\vec v_p)$ with $\{ v_0,\ldots, v_p\}$ a $p$-simplex of $B_n^m(\OO)$. As before, the $i$th face map forgets the $i$th vector. Maazen \cite[Section III.4, Theorem 4.2, and Corollary 4.5]{Maazen} proved that the barycentric subdivision of $||\mr{O}_n^m(\OO)_\bullet||$ is $(n-2)$-connected and hence $||\mr{O}_n^m(\OO)_\bullet||$ is $(n-2)$-connected. Consider the natural projections:
\begin{align*}
\mr{O}_n^m(\OO)_\bullet & \longrightarrow B_n^m(\OO)_\bullet \\
(\vec v_0, \vec v_1, \ldots, \vec v_p) & \longmapsto (\mr{span}(\vec v_0), \mr{span}(\vec v_1), \ldots, \mr{span}(\vec v_p) ).
\end{align*}
By picking a representative $\vec v$ for all lines $v$, one can construct a splitting of this map; see \cite[Proposition 2.13]{MPWY}. Thus, $||B_n^m(\OO)_\bullet||$ and hence $B_n^m(\OO)$ is $(n-2)$-connected.
\end{proof}



\subsubsection{Augmented partial frames} One of the innovations of Church--Putman \cite{CP} is the introduction of a simplicial complex of augmented partial frames, obtained by adding new ``additive'' simplices to $B_n^m(\OO)$. These correspond to certain relations in Steinberg modules.

\begin{definition}An \emph{augmented partial frame} is an unordered collection of lines $v_0,\ldots,  v_p$ such that, possibly after re-indexing,  $v_1, \ldots,  v_p$ is a partial frame and there are units $u_1,u_2$ in $\OO$ so that $\vec v_0=u_1 \vec v_1 + u_2 \vec v_2$.\end{definition}

\begin{definition}\label{def:augmented-frames} Let $M$ be a finite-rank free $\OO$-module. The \emph{complex of augmented partial frames} $BA(M)$ is the simplicial complex with $p$-simplices given by the union of the set of partial frames in $M$ of cardinality $p+1$ and the set of augmented partial frames in $M$ of cardinality $p+1$. A set $ w_0,\ldots,  w_q$ is a face of $ v_0,\ldots,  v_p$ if and only if $\{ w_0,\ldots,  w_q \} \subseteq \{ v_0,\ldots,  v_p \}$. We generally write $BA_n(\OO)$ for $BA(\OO^n)$. \end{definition}

We adapt the following notation from Church--Putman \cite[Definitions 4.7 and 4.11]{CP}.

\begin{definition} Fix $n>0$. Let $BA_n^m(\OO)$ be the  full subcomplex of \[\mr{Link}_{BA_{n+m}(\OO)}( e_1,\ldots,  e_m)\]
of simplices $ v_0,\ldots,v_p$ such that $v_i \not\subseteq \mr{span}(\vec e_1,\dots, \vec e_m)$ for all $i$.\end{definition}

\begin{definition}For $\sigma=\{w_0,\ldots,w_q\}$, let $\Linkh_{BA_n^m(\OO)}(\sigma)$ denote the full subcomplex of $\Link_{BA_n^m(\OO)}(\sigma)$ of simplices $\{v_0,\dots,v_p\}$ such that for all $i$, \[v_i \not\subseteq \mr{span}(\vec e_1,\dots, \vec e_m,\vec w_0,\dots,\vec w_q).\]
\end{definition}

Observe that $BA_n^m(\OO)$ is equal to $\Linkh_{BA_{n+m}(\OO)}(e_1,\dots,e_m)$. The span of the lines in a $p$-simplex $\{v_0,\ldots,v_p\}$ of $BA_n^m(\OO)$ has rank $p+1$ or $p$. The following definition, analogous to \cite[Definition 4.9]{CP}, describes the latter type of simplices:

\begin{definition} \label{DefnAdditive} Let $\sigma = \{v_0,\ldots,v_p\}$ be a $p$-simplex of $BA_n^m(\OO)$. 
	
	\begin{enumerate}[(i)]
		\item We say $\sigma$ is an \emph{internally additive} simplex if $\vec v_i =  u_k\vec v_k + u_j\vec v_j$ for some $i,j,k$ and some units $u_k,u_j$ in $\cO$.
		\item We say $\sigma$ is an \emph{externally additive} simplex if $\vec v_i =  u_k\vec e_k + u_j\vec v_j$ for some $i,j,k$ and some units $u_k,u_j$ in $\cO$. 
		
		More generally, if $\sigma$ is a simplex in $\Linkh_{BA_n^m(\OO)}(\{w_0, \ldots, w_p\})$, we say that $\sigma$ is an \emph{externally additive} simplex if $\vec v_i = u_k \vec v_k + u \vec w$ for some units $u_k,u$ in $\cO$, and primitive vector $\vec w$ spanning $e_1, \ldots, e_m , w_0,\ldots,w_{p-1}$ or  $w_p$.
		\item We say that a simplex is \emph{additive} if it is externally or internally additive.
	\end{enumerate}
\end{definition}

We will also need a subcomplex of certain links with control on the last coordinate:

\begin{definition} \label{Defn:fF}
	Let $f \colon \OO^{n+m} \m \OO$ denote the projection onto the last coordinate. If $\OO$ comes equipped with a preferred multiplicative Euclidean function $|-|$, we let $F(v)=|f(\vec v)|$ for $v$ a line spanned by $\vec v$. This is well-defined since $| - |$ is multiplicative. If $\OO$ is the Gaussian integers or Eisenstein integers, we will take $|-|$ to be the usual norm on the complex numbers, $|a+bi|=\sqrt{a^2+b^2}$. We will occasionally denote $|f(\vec v)|$ by $F(\vec v)$ instead of $F(v)$.
\end{definition} 

\begin{definition}Let $\sigma$ be a simplex of $BA^m_n(\OO)$. Then 
	\[\Linkh{}^<_{BA_n^m(\OO)}(\sigma)\]
	denotes the full subcomplex of $\Linkh_{BA_n^m(\OO)}(\sigma)$ consisting of lines $v$ such that there is a vertex $w$ in $\sigma$ with $F(v)<F(w)$. Similarly define 	\[\Link^{<}_{B_n^m(\OO)}(\sigma) \qquad \text{and} \qquad \Link^{<}_{BA_n^m(\OO)}(\sigma)\] for simplices $\sigma$ in $B_n^m(\OO)$ or $BA_n^m(\OO)$ respectively.
\end{definition}

We can deduce a connectivity result for the latter complex from that for $B^m_n(\OO)$. 

\begin{lemma}\label{LinkB} Let $\OO$ be a Euclidean domain and let $\sigma$ be a simplex in $B_n^m(\OO)$. Assume that $F(w)>0$ for some vertex $w$ in $\sigma$. Then $\Link^{<}_{B_n^m(\OO)}(\sigma)$ is Cohen--Macaulay of dimension $(n-\dim(\sigma)-2)$.\end{lemma}

\begin{proof}For a simplex $\tau$ in $\Link^{<}_{B_n^m(\OO)}(\sigma)$, 
	\[\Link_{\Link^{<}_{B_n^m(\OO)}(\sigma)}(\tau) = \Link^{<}_{B_n^m(\OO)}(\sigma \ast \tau),\]
	where $\ast$ denotes simplicial join; this uses the assumption that for all $v \in \tau$ we have $F(v)<F(w)$ for some $w \in \sigma$. Our goal is therefore to show that $\Link^{<}_{B_n^m(\OO)}(\sigma)$ is $(n-\dim(\sigma)-3)$-connected. 
	
	This result is proved in the case $\OO=\Z$ in Church--Putman \cite[Lemma 4.5]{CP}, and their proof adapts readily to all Euclidean domains. For completeness, we sketch a proof here. Given \autoref{thm:maazen}, it suffices to show that there is a simplicial retraction 
	\[\pi \colon \Link_{B_n^m(\OO)}(\sigma) \longrightarrow  \Link^{<}_{B_n^m(\OO)}(\sigma).\]
	Let $x$ be a vertex of $\sigma$ with $M=F(x)$ maximal among the vertices, and fix a representative vector $\vec x$.  We define the map $\pi$ on vertices of $ \Link_{B_n^m(\OO)}(\sigma) $ as follows. For $ v \in \Link_{B_n^m(\OO)}(\sigma) $, choose a representative $\vec v \in \OO^{n+m}$, and let $q_v \in \OO$ be a quotient of $f(\vec v)$ on division by $f(\vec x)$, in the sense of the Euclidean algorithm. If $F( \vec v) < F( \vec x) $ we take $q_v=0$. Then by construction \[0 \leq F( \vec v - q_v \vec x) < F( \vec x) = M.\] We can thus define $\pi(v)$ to be the line spanned by $( \vec v - q_v \vec x)$. It is straightforward to verify that this map on vertices extends over simplices in $\Link_{B_n^m(\OO)}(\sigma)$, and fixes simplices in $\Link^{<}_{B_n^m(\OO)}(\sigma)$.
\end{proof} 

\subsection{New connectivity results}

By modifying the proof of Church--Putman  \cite[Theorem C']{CP}, we will prove the following.

\begin{theorem}\label{BAnOtherRings}
For $\OK$ the Gaussian integers or Eisenstein integers, the simplicial complexes $BA_n^m(\OK)$ are Cohen--Macaulay of dimension $n$ for all $n$ and $m$ satisfying $n \geq 1$ and $n+m \geq 2$.
\end{theorem}

Recall that  $BA_n(\OK)=BA_n^0(\OK)$ so the above theorem implies that $BA_n(\OK)$ is spherical. 

\subsubsection{Low dimensional cases}

Before we can prove that $BA_n^m(\OK)$ is Cohen--Macaulay for general $n$ and $m$, we will first need to study small values. The argument will be by induction, and the following will be the base case. 

\begin{lemma}\label{linkUnits}
Let $m \geq 1$. Then $BA_1^m(\OO)$ is connected if and only if $\OO$ is additively generated by multiplicative units. 
\end{lemma}

\begin{proof}  An $(m+1)$-simplex in $BA_{m+1}(\OO)$ containing $e_1,  e_2, \ldots, e_m$ is an augmented frame for $\OO^{m+1}$ of the form 
  \[\{e_1, e_2, \ldots, e_m, x, y\}\] where $\{ \vec e_1, \ldots, \vec e_m, \vec x\}$ is a basis for $\OO^{m+1}$, and $\vec y = u_1 \vec x + u_2 \vec e_j$ or $\vec y = u_1 \vec e_i + u_2 \vec e_j$  for some $i,j$ and some units $u_1, u_2 \in \OO$. 
By definition $BA_1^m(\OO)$ is the subgraph of the edges $\{x, y\}$ on those lines $x$, $y$ described above that are not contained in $\OO^m$. 

In particular, the vertices of this graph are the spans of vectors of the form $(x_1,\ldots,x_m,1)$ with $x_i \in \OO$; these vertices may be uniquely represented by a vector with $(m+1)$st coordinate equal to $1$, and the other coordinates may be any values in $\OO^m$.  There is an edge from $(x_1,\ldots,x_m,1)$ to $(y_1,\ldots,y_m,1)$ if and only if there is an $i$ such that $x_j=y_j$ for all $j \neq i$ and $x_i-y_i$ is a unit. See \autoref{CayleyZ2}. There is therefore a path from $(x_1,\ldots,x_m,1)$ to $(y_1,\ldots,y_m,1)$ if and only if each value $x_i-y_i$ is a sum of units. The claim follows. \end{proof}

\begin{figure}[!ht]  
\centering
\labellist
\Large \hair 0pt
\pinlabel { \color{black} $\tiny{  \displaystyle \begin{pmatrix} 1 \\ 1 \\1 \end{pmatrix} }$} [l] at 61 68
\pinlabel { \color{black} $\tiny{  \displaystyle \begin{pmatrix} 1 \\ 0 \\1 \end{pmatrix} }$} [l] at 61 46
\pinlabel { \color{black} $\tiny{  \displaystyle \begin{pmatrix} 1 \\ -1 \\1 \end{pmatrix} }$} [l] at 61 23
\pinlabel { \color{black} $\tiny{  \displaystyle \begin{pmatrix} 0 \\ -1 \\1 \end{pmatrix} }$} [l] at 39 23
\pinlabel { \color{black} $\tiny{  \displaystyle \begin{pmatrix} 0 \\ 0 \\1 \end{pmatrix} }$} [l] at 39 46
\pinlabel { \color{black} $\tiny{  \displaystyle \begin{pmatrix} 0 \\ 1 \\1 \end{pmatrix} }$} [l] at 39 68
\pinlabel { \color{black} $\tiny{  \displaystyle \begin{pmatrix} -1 \\ -1 \\1 \end{pmatrix} }$} [l] at 16 23
\pinlabel { \color{black} $\tiny{  \displaystyle \begin{pmatrix} -1 \\ 0 \\1 \end{pmatrix} }$} [l] at 16 46
\pinlabel { \color{black} $\tiny{  \displaystyle \begin{pmatrix} -1 \\ 1 \\1 \end{pmatrix} }$} [l] at 16 68
\endlabellist
\includegraphics[scale=2]{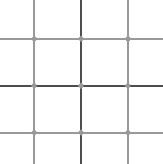} 
\caption{Part of the complex $BA_1^2(\Z)$.}
\label{CayleyZ2}
\end{figure}
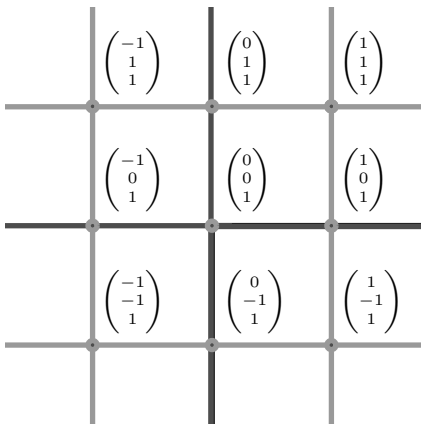 

As illustrated in \autoref{CayleyZ2}, if $\OO$ is additively generated by units, then $BA_1^m(\OO)$ is the Cayley graph for the additive group $\OO^m$ associated to the generating set $\{u \vec e_i \; | \; \text{$u$ a unit, $i=1, \ldots, m$}\}$.

\begin{proposition} \label{BA2}
Let $\OK$ denote the Gaussian integers or Eisenstein integers. Then $BA_2(\OK)$ is Cohen--Macaulay of dimension 2.
\end{proposition}

\begin{proof}For $BA_2(\OK)$ to be Cohen--Macaulay of dimension 2, no condition is imposed on the link of $2$-simplices. The link of a $1$-simplex $\{v_1,v_2\}$ always contains the vertex $\mr{span}(\vec{v}_1+\vec{v}_2)$, so is $(-1)$-connected. The link of a vertex $v$ is isomorphic to $BA^1_1(\OK)$, which is $0$-connected by \autoref{linkUnits} since $\OK$ is additively generated by units. Thus it remains to show that $BA_2(\OK)$ is $1$-connected.
	
First observe that $BA_2(\OK)$ is $0$-connected since its $1$-skeleton is $B_2(\OK)$, which is connected by \autoref{thm:maazen}. Now suppose $\phi \colon S^1 \to BA_2(\OK)$ is a simplicial map with respect to some simplicial structure on $S^1$. Our goal is to show that $\phi$ is homotopic to a constant map. 

Let $$M \coloneqq \max_{\text{vertices } x \in S^1}\{F(\phi(x))\}.$$  Here $F$ is as in \autoref{Defn:fF}. If $M=0$, then $\phi$ is the constant map  at $e_{1}$. Hence assume that $M>0$ and that $\phi$ is not constant. We will prove that $\phi$ can be homotoped to a map $\hat \phi$ having one less vertex mapping to a line $w$ with $F(w)=M$. As it is not constant, by collapsing edges we may assume that $\phi$ is simplex-wise injective. Pick a vertex $x \in S^1$ with $F(\phi(x)) = M$. Let $x_1$ and $x_2$ be the vertices adjacent to $x$, and let $\phi(x)=w$ and $\phi(x_i)=v_i$ denote the images in $BA_2(\OK)$.

We will homotope $\phi$ to make the vertices adjacent to $x$ have last coordinate $<M$, if they do not already. So suppose without loss of generality that $F(v_1)=M$. By \autoref{lem0}, we can find a unit $u$ with $F(\vec v_1-u \vec w)<M$. 
Since we assume that $\phi$ is simplex-wise injective, $v_1 \neq w$ and hence $\vec v_1-u \vec w \neq 0$. Let $v_1'$ denote the span of $\vec v_1-u \vec w$ and note that $\{v_1,v_1',w\}$ forms a simplex in $BA_2(\OK)$. Thus, there is a homotopy $H_t$ with:
\begin{itemize}
\item  $H_0=\phi$,
\item $H_t(y)=\phi(y)$ for $y \in S^1$ not in the interior of $\{x_1,x\}$,
\item $H_1$ is a simplicial map from $S^1$ with the edge $\{x_1,x\}$ subdivided once and with middle vertex being mapped to $v'_1$.
\end{itemize}
Since $F(v_1')<M$, we see that $\phi$ is homotopic to a simplicial map where one of the vertices adjacent to $x$ has last coordinate smaller than $M$. If necessary, we can also apply this procedure to $v_2$. 

Thus, we may assume that the images of vertices adjacent to $x$ are $v_1,v_2 \in \Link{}^<_{BA_2(\OK)}(w)$. Pick representatives $\vec v_1$ and $\vec w$. The vertices of \[\Link{}^<_{BA_2(\OK)}(w)\] are precisely the spans of vectors of the form $\vec v_1+a \vec w$ with $a \in \OK$ such that $F(\vec v_1+a \vec w)<F(w) = M$, equivalently, such that 
\[\left| \frac{f(\vec v)}{f(\vec w)}+a\right|<1.\]
There is an edge between the span of $\vec v_1+a \vec w$ and the span of $\vec v_1+b \vec w$ if and only if $a-b$ is a unit.  Thus, $\Link{}^<_{BA_2(\OK)}(w)$ is isomorphic to the graph $G_{z}$ of \autoref{lem1} for $z=-f(\vec v)/f(\vec w)$. Hence, that lemma implies it is connected. Let $A \cong S^0$ denote the set containing $x_1$ and $x_2$. Since the target is connected, the map \[\phi|_A \lra \Link{}^<_{BA_2(\OK)}(w)\] is null-homotopic. Let 
\[g \colon \mr{Cone}(A) \lra \Link{}^<_{BA_2(\OK)}(w)\]
be a null-homotopy from a choice of simplicial complex structure on $\mr{Cone}(A)$; note that such a simplicial structure is just a subdivision of an interval. See \autoref{ConedSphere}.

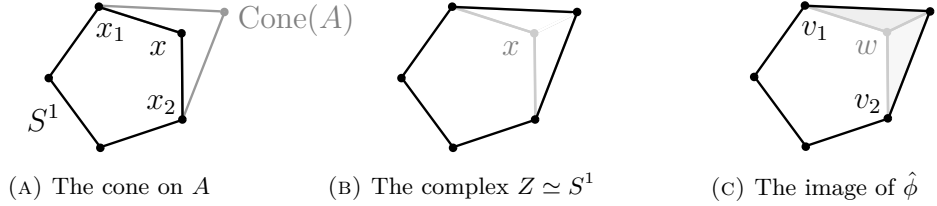
\begin{figure}[h!]
\begin{subfigure}{.3\textwidth} \centering
\labellist
\Large \hair 0pt
\pinlabel { \color{black} $\small{  \displaystyle S^1 }$} [tr] at 5 15
\pinlabel { \color{gray} $\small{  \mr{Cone}(A) }$} [tl] at 58 47
\pinlabel { \color{black} $\small{  x }$} [l] at 30 33
\pinlabel { \color{black} $\small{  x_1 }$} [l] at 13 38
\pinlabel { \color{black} $\small{  x_2 }$} [l] at 29 15
\endlabellist
\qquad
\includegraphics[scale=1.2]{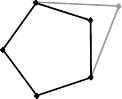}   \\[.2em]
\subcaption{The cone on $A$} 
\label{ConedSphere}
\end{subfigure} 
\begin{subfigure}{.3\textwidth} \centering
\labellist
\Large \hair 0pt
\pinlabel { \color{gray} $\small{  x }$} [l] at 30 33
\endlabellist
\qquad
\includegraphics[scale=1.2]{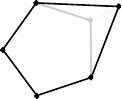}   \\[.2em]
\subcaption{The complex $Z \simeq S^1$} 
\label{ConedSphereBoundary-Empty}
\end{subfigure} 
\begin{subfigure}{.3\textwidth} \centering
\labellist
\Large \hair 0pt
\pinlabel { \color{gray} $\small{  w }$} [l] at 30 33
\pinlabel { \color{black} $\small{  v_1 }$} [l] at 13 38
\pinlabel {\color{black} $\small{  v_2}$} [l] at 33 15
\endlabellist
\qquad
\includegraphics[scale=1.2]{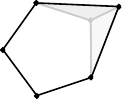}   \\[.2em]
\subcaption{The image of $\hat \phi$} 
\label{ConedSphereBoundary}
\end{subfigure} 
\caption{Sample illustrations of the homotopy between $\phi$ and $\hat \phi$} 
\end{figure}

Let $Z$ be $S^1$ with simplicial structure given by replacing $\{x_1,x\} \cup \{x,x_2\}$ with $\mr{Cone}(A)$, as in \autoref{ConedSphereBoundary-Empty}. Let $\hat \phi \colon Z \m BA_2(\OK)$ be given by the formula: \[\hat \phi(y) =
  \begin{cases}
    \phi(y) & \text{ if } y \in Z \setminus \mr{Cone}(A), \\

 g(y) & \text{ if } y \in  \mr{Cone}(A). 
   \end{cases}\] 
   Observe that $\phi$ maps $\{x_1,x\} \cup \{x,x_2\}$ into $\mr{Star}_{BA_2(\OK)}(w)$, $\hat \phi$ maps $\mr{Cone}(A)$ into $\mr{Star}_{BA_2(\OK)}(w)$, and both maps agree on $A$. See  \autoref{ConedSphereBoundary}. Thus, $\phi$ and $\hat \phi$ are homotopic. The new map $\hat \phi$ maps one fewer vertex to a vertex realizing the  value $M$. Iterating this procedure produces a map with image contained in the star of $e_{2}$.
   \end{proof}

The following general lemma is well-known. See \autoref{Posection} for a review of some notation related to posets. When we refer to the connectivity of a poset we mean the connectivity of the geometric realization of its nerve, and similarly for maps between posets.

\begin{lemma}\label{poset1conn}
Let $p \colon X \m Y$ be a map of simplicial complexes with $Y$ Cohen--Macaulay of dimension $n$. Suppose for each simplex $\sigma$ of $Y$ the inverse image $p^{-1}(\sigma)$ is $(\dim(\sigma)-1)$-connected. Then $X$ is $(n-1)$-connected.
\end{lemma}

\begin{proof}Let $\cat{simp}(X)$ be the poset of simplices of $X$, ordered by inclusion, and similarly for $\cat{simp}(Y)$. Then $p$ induces a functor $\cat{simp}(p) \colon \cat{simp}(X) \to \cat{simp}(Y)$ and there are homeomorphisms making the following diagram commute:
	\[\begin{tikzcd}{|\cat{simp}(X)|} \rar{\cong} \dar[swap]{|\cat{simp}(p)|} &[10pt] X \dar{p} \\[10pt]
	{|\cat{simp}(Y)|} \rar{\cong} & Y.\end{tikzcd}\]
These homeomorphisms are just the standard homeomorphisms between a simplicial complex and its barycentric subdivision. Thus it suffices to prove that $|\cat{simp}(p)|$ is $n$-connected. We now apply \cite[Corollary 2.2]{vdKallenLooijenga} (also see \cite[Theorem 4.1]{e2cellsIII}): a map \[f \colon \mathbf{X} \lra \mathbf{Y}\] of posets is $n$-connected if there is a function $t \colon \mathbf{Y} \to \bZ$ such that for every $y \in \mathbf{Y}$ we have \begin{itemize} \item $\mathbf{Y}_{>y} \coloneqq \{y' \in \mathbf{Y} \mid y'>y\}$ is $(n-t(y)-2)$-connected, \item $f/y \coloneqq \{x \in  \mathbf{X} \mid f(x) \leq y\}$ is $(t(y)-1)$-connected.  \end{itemize}

\noindent Here we take $n$ as in the hypothesis, and $t(\sigma) = \dim(\sigma)$. Then $\cat{simp}(p)/\sigma$ is $\cat{simp}(p^{-1}(\sigma))$ which by assumption is $(\dim(\sigma)-1)$-connected. Similarly, $\cat{simp}(Y)_{>\sigma}$ is $\cat{simp}(\mr{Link}_Y(\sigma))$ which is  $(n-\dim(\sigma)-2)$-connected because $Y$ is Cohen--Macaulay of dimension $n$.
\end{proof}

\begin{proposition} \label{BA21}
Let $\OK$ denote the Gaussian integers or Eisenstein integers. Let $w$ be a line in $\OK^3$ with $F(w)>0$. Then $\Link{}^<_{BA_3(\OK)}(w)$ is $1$-connected. 
\end{proposition}

\begin{proof}
Let $w_1,w_2$ be lines with $\{w_1,w_2,w\}$ a simplex in $B_3(\OK)$. Let $L \colon \OO^3 \m w_1 \oplus w_2$ be given by \[L(a \vec w_1+b \vec w_2 + c \vec w)=a \vec w_1+b \vec w_2.\] If $a \vec w_1+b \vec w_2 + c \vec w$ spans a line in the link of $w$, then $a \vec w_1+b \vec w_2$ spans a line which only depends on the line spanned by $a \vec w_1+b \vec w_2 + c \vec w$.
  We will show that $L$ induces a simplicial map \[p \colon \Link{}^<_{BA_3(\OK)}(w) \lra BA(w_1 \oplus w_2)\] given by sending the span of $a \vec w_1+b \vec w_2 + c \vec w$ to the span of $a \vec w_1+b \vec w_2$. 
  
  \vspace{.5em}
\noindent {\bf Claim: $p$ is a simplicial map.} 
\vspace{.5em}

  Since the span of $a \vec w_1+b \vec w_2 + c \vec w$ is a vertex of $\Link_{BA_3(\OK)}(w)$, the span of $a \vec w_1+b \vec w_2 + c \vec w$ is not $w$ and so $a \vec w_1+b \vec w_2$ is nonzero. In fact, $a \vec w_1+b \vec w_2$ spans a line. Thus the formula for $p$ produces a map on sets of vertices. 

Now we check that it extends to a simplicial map, starting with $2$-simplices. Let $$\{v_0,v_1,v_2\} \in \Link{}^<_{BA_3(\OK)}(w)$$ be an internally additive 2-simplex in the sense of \autoref{DefnAdditive}. Reorder and pick representatives so that $\vec v_0=\vec v_1+\vec v_2$. Since $\{v_1,v_2\}$ is an edge of $\Link_{BA_3(\OK)}(w)$ and $\{v_0,v_1,v_2\}$ is internally additive, $w,v_1, v_2$ are a partial frame and thus $p(v_1),p(v_2)$ is also a partial frame. Since $L$ is linear, $L(\vec v_0)=L(\vec v_1)+L(\vec v_2)$. Thus, $\{p(v_0),p(v_1),p(v_2)\}$ forms a simplex in $BA(w_1 \oplus w_2)$. 

Next, let $$\{v_0',v_1',v_2'\} \in \Link{}^<_{BA_3(\OK)}(w)$$ be an externally additive $2$-simplex. Pick representatives and reindex so that $\vec v_0'=\vec v_1'+\vec w$. Then $p(\vec v_0')=p(\vec v_1')$, and $p$ maps $\{v_0',v_1',v_2'\}$ to $\{p(v_1'),p(v_2')\}$ which forms a 1-simplex in $BA(w_1 \oplus w_2)$. Since every simplex of $\Link{}^<_{BA_3(\OK)}(w)$ is contained in one of these two types of simplices, we have checked that $p$ is a simplicial map.

\vspace{.5em}
\autoref{BA2} implies that $BA(w_1 \oplus w_2)$ is Cohen--Macaulay of dimension 2. Thus, to apply \autoref{poset1conn} to $p$ with $n=2$, it suffices to show the fiber over a simplex $\sigma \subseteq BA(w_1 \oplus w_2)$ is $(\dim(\sigma)-1)$-connected. 

\vspace{.5em}
\noindent {\bf Claim: If $\dim(\sigma)=0$, then $p^{-1}(\sigma)$ is $(-1)$-connected (in fact, connected).} 
\vspace{.5em}

\noindent Let $\sigma=\{v_0\}$. Fix a representative $\vec v_0$ spanning $v_0$ and $\vec w$ spanning $w$. Note that $p^{-1}(\{v_0\})$ has vertices of the form $\vec v_0+a \vec w$ such that $a \in \OK$ and $F(\vec v_0+a \vec w)<F(w)$. There is an edge between $\vec v_0+a \vec w$ and $\vec v_0+b \vec w$ if and only if $a-b$ is a unit.  The constraint $F(\vec v_0+a \vec w)<F(w)$ is a constraint on $a$ equivalent to the condition that \[\left|a+\frac{f(\vec v_0)}{f(\vec w)}\right|<1.\]
Thus, $p^{-1}(\{v_0\})$ is isomorphic to a subgraph of the Cayley graph of $\OK$ with units as generators. Specifically, $p^{-1}(\{v_0\})$ is the subgraph on those vertices contained in the open ball of radius $1$ (in the complex metric) around the complex number $-f(\vec v_0)/f(\vec w)$. This subgraph is $0$-connected by \autoref{lem1} and hence is also $(-1)$-connected.

\vspace{.5em}
\noindent {\bf Claim: If $\dim(\sigma)=1$, then $p^{-1}(\sigma)$ is connected (in fact, $1$-connected).} 
\vspace{.5em}

\noindent Let $\sigma=\{v_0,v_1\}$. Then the vertices of $p^{-1}(\{v_0,v_1\})$ are lines spanned by vectors of the form
$$ \vec v_0+a \vec w, \quad \vec v_1+b \vec w \qquad $$ 
with $a, b \in \OK$ subject to the condition that $F(\vec v_0+a \vec w),F(\vec v_1+b \vec w) <F(w)$. The edges correspond to pairs of these vectors of the form
$$ \{\vec v_0+a \vec w, \; \vec v_1+b \vec w\}, \quad  \{\vec v_0+a \vec w,\; \vec v_0+(a+u) \vec w\},  \quad \{\vec v_1+b \vec w,\; \vec v_1+(b+u) \vec w\}$$ 
with $u$ a unit in $\OK$, and the $2$-simplices correspond to triples of these vectors
$$ \{\vec v_0+a \vec w, \; \vec v_1+b \vec w,\; \vec v_0+(a+u) \vec w\}, \quad   \{\vec v_0+a \vec w, \; \vec v_1+b \vec w, \;\vec v_1+(b+u) \vec w\},$$ 
with $u$ a unit. 
Notably, $p^{-1}(\{v_0,v_1\})$ is the $2$-skeleton of the join of $p^{-1}(\{v_0\})$ and $p^{-1}(\{v_1\})$. Since $p^{-1}(\{v_0\})$ and $p^{-1}(\{v_1\})$ are connected, $p^{-1}(\{v_0,v_1\})$ is $1$-connected and hence $0$-connected.

\vspace{.5em}
\noindent {\bf Claim: If $\dim(\sigma)=2$, then $p^{-1}(\sigma)$ is $1$-connected.} 
\vspace{.5em}
 
\noindent Let $\sigma=\{v_0,v_1,v_2\}$, so $\sigma$ is necessarily an augmented frame. Let \[X= p^{-1}(\{v_0,v_1\}) \cup p^{-1}(\{v_0,v_2\}) \cup p^{-1}(\{v_1,v_2\}) \quad \subseteq \quad p^{-1}(\{v_0,v_1,v_2\}).\] Since $X$ contains the $1$-skeleton of $p^{-1}(\{v_0,v_1,v_2\})$, \[\pi_i(X,x_0) \lra \pi_i(p^{-1}(\{v_0,v_1,v_2\}),x_0)\] is surjective for $i=0,1$ and all basepoints $x_0$. Observe that \[p^{-1}(\{v_i,v_j\}) \cap p^{-1}(\{v_i,v_k\}) = p^{-1}(\{v_i\})\] if $j \neq k$, and the inclusions of the intersection into each term is the inclusion of a subcomplex. This implies that $X$ is connected and hence so is $p^{-1}(\{v_0,v_1,v_2\})$.

Our next goal is to pick basepoints $x_i \in p^{-1}(\{v_i\})$ such that $\{x_0,x_1,x_2\}$ forms a simplex in $p^{-1}(\{v_0,v_1,v_2\})$. Pick representatives for $v_i$ and $w$ such that $\vec v_0=\vec v_1+\vec v_2$. As noted before, a representative for a line in $ p^{-1}(\{v_i\})$ is a vector of the form $\vec v_i+r_i  \vec w$ with $r_i \in \OK$ such that  \[\left|r_i+\frac{f(\vec v_i)}{f(\vec w)}\right|<1.\] By \autoref{lem2G} and \autoref{lem2E} applied to $z_1=-\frac{f(\vec v_1)}{f(\vec w)}$ and $z_2=-\frac{f(\vec v_2)}{f(\vec w)}$, we see that we can find $r_0,r_1,r_2$ with $r_0=r_1+r_2$ and with \[x_i:=\mr{span}(\vec v_i+r_i \vec w) \in  p^{-1}(\{v_i\}) \quad \text{for $i=0,1$ and $2$.}\] Since \[(\vec v_0+r_0 \vec w)=(\vec v_1+r_1 \vec w)+(\vec v_2+r_2 \vec w), \] the lines $\{x_0,x_1,x_2\}$ form a simplex.



Since each $p^{-1}(\{v_i,v_j\})$ is 1-connected and each $p^{-1}(\{v_i\})$ is 0-connected, a groupoid version of the Seifert--van Kampen theorem (see e.g. May \cite[Chapter 2, Section 7]{concise}) then implies that $\pi_1(X,x_0)$ is generated by any loop that is a concatenation of \begin{itemize}
 \item a path $\gamma_{01}$ in $p^{-1}(\{v_0,v_1\})$ 
 from $x_0$ to $x_1$,
 \item a path $\gamma_{12}$ in $p^{-1}(\{v_1,v_2\})$ 
 from $x_1$ to $x_2$,
 \item a path $\gamma_{20}$ in $p^{-1}(\{v_0,v_2\})$ 
 from $x_2$ to $x_0$.
\end{itemize}

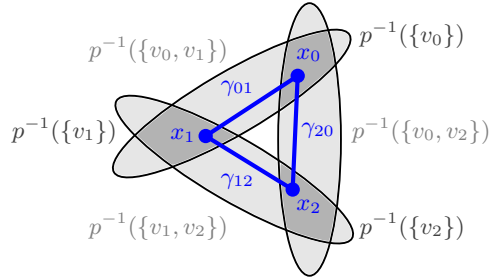
\begin{figure}[!ht]  
\centering
\labellist
\Large \hair 0pt
\pinlabel {\color{blue} \small{ $ x_1$}} [l] at 100 280
\pinlabel {\color{blue} \small{ $ x_0$}} [l] at 350 440
\pinlabel {\color{blue} \small{ $ x_2$}} [l] at 350 140
\pinlabel {\color{blue} \small{ $ \gamma_{20}$}} [l] at 360 290
\pinlabel {\color{blue} \small{ $ \gamma_{01}$}} [l] at 200 375
\pinlabel {\color{blue} \small{ $ \gamma_{12}$}} [l] at 200 190
\pinlabel {\color{gray} \small{ $p^{-1}(\{v_0,v_2\})$}} [l] at 460 290
\pinlabel {\color{gray} \small{ $p^{-1}(\{v_1,v_2\})$}} [r] at 230 100
\pinlabel {\color{gray} \small{ $p^{-1}(\{v_0,v_1\})$}} [r] at 230 450
\pinlabel {\color{darkgray} \small{ $p^{-1}(\{v_2\})$}} [r] at 700 100
\pinlabel {\color{darkgray} \small{ $p^{-1}(\{v_0\})$}} [r] at 700 480
\pinlabel {\color{darkgray} \small{ $p^{-1}(\{v_1\})$}} [r] at 10 290
\endlabellist
\includegraphics[scale=.2]{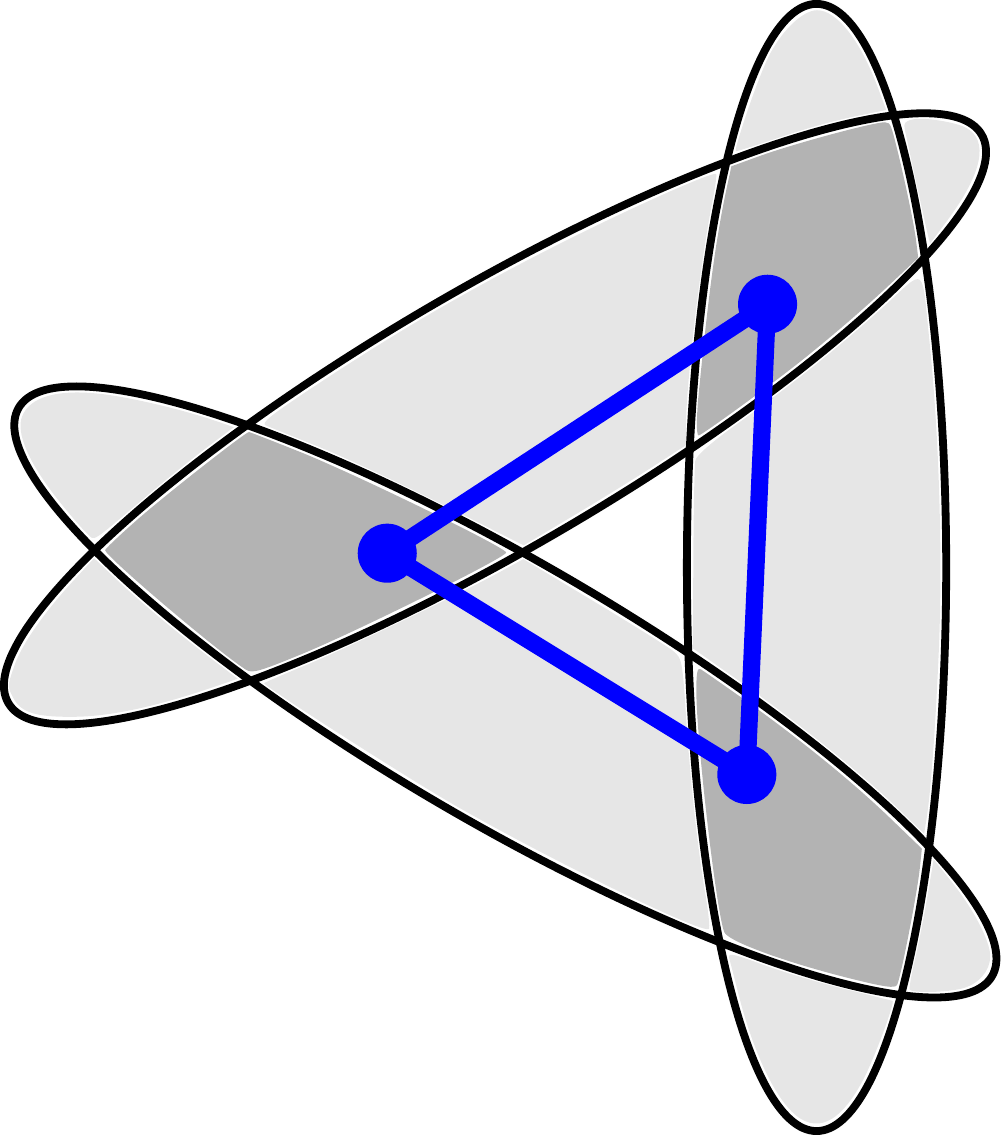} 
\caption{A schematic of a generator of $\pi_1(X,x_0)$.}
\label{VanKampen}
\end{figure} 
\noindent Pick each of these paths from $x_i$ to $x_j$ to be the edge from $x_i$ to $x_j$. See \autoref{VanKampen}. Our specific choice of $x_0,x_1,x_2$ from the previous paragraph then implies that $$\pi_1(X,x_0) \lra \pi_1(p^{-1}(\{v_0,v_1,v_2\},x_0)$$ is the zero map since $\{x_0,x_1,x_2\}$ forms a simplex. Since $\pi_1(X,x_0) \m \pi_1(p^{-1}(\{v_0,v_1,v_2\},x_0)$ is also surjective, this implies that $\pi_1(p^{-1}(\{v_0,v_1,v_2\}),x_1)$ is trivial. Since we already showed that $p^{-1}(\{v_0,v_1,v_2\})$ is connected, this completes the argument.
\end{proof}

\subsubsection{General case} Having completed these low-dimensional cases, we proceed to prove that the complex of augmented partial bases is spherical for general $n$ and $m$. 

\begin{lemma} \label{RETLEM} Let $\OK$ denote the Gaussian integers or Eisenstein integers. Let $w$ be a line with $F(w) \neq 0$. The inclusion 
	\[\iota \colon \Linkh{}^<_{BA_n^m(\OK)}(w) \lra \Linkh_{BA_n^m(\OK)}(w)\]
admits a (not necessarily simplicial) retraction.\end{lemma}

\autoref{RETLEM} is an analogue of Church--Putman \cite[Proposition 4.17]{CP}. We first define the retraction $\pi$ on the vertices of $\Linkh_{BA_n^m(\OK)}(w)$. Unfortunately, this map on vertices does not extend to a simplicial map. There are certain $1$- and $2$-simplices $\sigma$, the \emph{carrying} simplices, for which the images of the vertices do not span a simplex in $\Linkh{}^<_{BA_n^m(\OK)}(w)$, although every face of $\sigma$ does map to a simplex.   As in Church--Putman, we remedy this problem by changing the simplicial structure on $\Linkh_{BA_n^m(\OK)}(w)$. Specifically, we subdivide each carrying simplex $\sigma$ -- leaving its boundary unaltered -- and correspondingly subdivide each simplex that has $\sigma$ as a face. We may then define the map $\pi$ on the new simplicial structure. Although $\pi$ does not respect the simplicial structure of $\Linkh_{BA_n^m(\OK)}(w)$, it does define a topological retraction. We note that to prove \cite[Proposition 4.17]{CP}, Church--Putman subdivided carrying triangles by inserting a single vertex. In the case of the Gaussian integers or Eisenstein integers, this is not possible. Instead, we will use a more elaborate connectivity argument.

%
%

\begin{proof}[Proof of \autoref{RETLEM}]
Recall the functions $f$ and $F$ of \autoref{Defn:fF}. For each vertex $v \in \Linkh_{BA_n^m(\OK)}(w)$, pick a vertex $v^{\pi} \in \Linkh{}^<_{BA_n^m(\OK)}(w)$ such that: \begin{enumerate}[(i)]
\item \label{i} There are representatives $\vec v$ and $\vec w$ and $a \in \OK$ such that $v^{\pi}$ is the span of $\vec v+a \vec w$,
\item $F(v^{\pi}) <F(w)$,
\item $v^{\pi}=v$ if $F(v) < F(w)$. \label{iii}
 
\end{enumerate} Such an assignment exists because of the Euclidean algorithm. Specifically, we choose $a$ so that $f(\vec v+a \vec w)$ is a least residue of $f(\vec v)$ modulo $f(\vec w)$. 

If the vectors $\vec w, \vec v_0, \vec v_1, \vec v_2, \ldots$ form a partial basis, then so too will the vectors $$\vec w, \;\; \vec v_0^{\; \pi}=\vec v_0+a_0\vec w, \;\; \vec v_1^{\; \pi} = \vec v_1 + a_1 \vec w, \;\; \vec v_2^{\; \pi} = \vec v_2 + a_2 \vec w, \ldots$$ Unfortunately, if a triple of vectors $\vec v_0, \vec v_1, \vec v_2$ satisfies $ \vec v_0 =  \vec v_1 +   \vec v_2$, there is no reason that the same linear relation will hold amongst representatives of their images $ \vec v_0^{\; \pi}, \vec v_1^{\; \pi}, \vec v_2^{\; \pi}$. (\autoref{lem2G} and \autoref{lem2E} imply that, for a fixed such triple  $\vec v_0, \vec v_1, \vec v_2$, we could choose least residues to arrange that $  \vec v_0^{\; \pi} =  \vec v_1^{\; \pi} +  \vec v_2^{\; \pi}$. However, there is no way to choose an image $v^{\pi}$ for each line $v$ to preserve all such linear relations simultaneously.) Similarly, if $ \vec v_0 =  \vec v_1 + \vec w$, the same relation need not hold amongst their images. Consequently,  the assignment on vertices $v \mapsto v^{\pi}$ does not extend over simplices in $\Linkh_{BA_n^m(\OK)}(w)$. 
 
Following Church--Putman  \cite[Proof of Proposition 4.17]{CP}, we call an internally additive $2$-simplex $\sigma=\{v_0,v_1,v_2\}$ \emph{carrying} if $\{v_0^\pi,v_1^\pi,v_2^\pi\}$ does not form a $2$-simplex. Let $\Link{}^<_{BA(v_0 \oplus v_1 \oplus w)  }(w)$  denote $\Link{}_{BA(v_0 \oplus v_1 \oplus w)  }(w)  \cap \Linkh{}^<_{BA_n^m(\OK)}(w) $. Observe that the edges $\{v_0^\pi,v_1^\pi\}$, $\{v_1^\pi,v_2^\pi\}$, $\{v_2^\pi,v_0^\pi\}$ form simplices. Thus, the union of these three edges forms a loop in $ \Link{}^<_{BA(v_0 \oplus v_1 \oplus w)  }(w)$, which we denote by $\gamma_\sigma$. We will use \autoref{BA21} to deduce that this loop is null-homotopic in $\Link{}^<_{BA(v_0 \oplus v_1 \oplus w)  }(w)$ once we describe an isomorphism \[\Link{}^<_{BA(v_0 \oplus v_1 \oplus w)  }(w) \cong \Link{}^<_{BA_3}(\tilde w)\] for some line $\tilde w \subset \OK^3$. Let $g:v_0 \oplus v_1 \oplus w \m \OK$ be the restriction of $f$. The image of $g$ is a nonzero principal ideal of $\OK$, say $(a)$. Let $g':v_0 \oplus v_1 \oplus w \m \OK$ be given by $g'(\vec v)=g(\vec v)/a$. Note that $|g(\vec v)|<|g(\vec w)|$ if and only if $|g'(\vec v)|<|g'(\vec w)|$. Since $g'$ is surjective, we may pick an isomorphism $\phi:v_0 \oplus v_1 \oplus w \m \OK^3$ which identifies $g'$ with projection onto the last coordinate. The isomorphism $\phi$ identifies $\Link{}^<_{BA(v_0 \oplus v_1 \oplus w)  }(w)$ with $\Link{}^<_{BA_3}(\tilde w)$ for $\tilde w=\phi(w)$. Thus, $\gamma_\sigma$  is null-homotopic in $\Link{}^<_{BA(v_0 \oplus v_1 \oplus w)  }(w)$.

 For each carrying internally additive $2$-simplex $\sigma=\{v_0,v_1,v_2\}$, pick  a simplicial map
\[H_{\sigma} \colon T(\sigma) \lra \Link{}^<_{BA(v_0 \oplus v_1 \oplus w)  }(w)\]
from a triangulation $T(\sigma)$ of the standard $2$-simplex, with $\gamma_{\sigma}$ equal to the restriction of $H_{\sigma}$ to the boundary of $T(\sigma)$. In particular, the triangulation $T(\sigma)$ does not subdivide the boundary triangle.

Similarly, call an externally additive $1$-simplex $\sigma=\{v_0,v_1\}$ \emph{carrying} if $\{v_0^\pi,v_1^\pi\}$ does not form a $1$-simplex. The vertices $v_0^\pi,v_1^\pi$ form a 0-sphere in $ \Link{}^<_{BA(v_0 \oplus w)  }(w)$, which we denote by $\gamma_{\sigma}$. In the proof of \autoref{BA2}, we identified $ \Link{}^<_{BA(v_0 \oplus w)  }(w)$ with a graph of the form $G_z$. Thus, by \autoref{lem1}, it is connected and so $\gamma_{\sigma}$ is null-homotopic. For each carrying externally additive $1$-simplex $\sigma=\{v_0,v_1\}$, pick  a simplicial map
\[H_{\sigma} \colon T(\sigma) \lra \Link{}^<_{BA(v_0 \oplus w)  }(w)\]
with $T(\sigma)$ a triangulation of the standard $1$-simplex, and $\gamma_{\sigma}$ equal to the restriction of $H_{\sigma}$ to the boundary of $T(\sigma)$.

Call a simplex \emph{carrying} if it is a carrying internally additive $2$-simplex or a carrying externally additive $1$-simplex. Let $X$ be obtained from $\Linkh_{BA_n^m(\OK)}(w)$ by replacing \[\mr{Star}_{\Linkh_{BA_n^m(\OK)}(w)}(\sigma ) \qquad \text{with} \qquad \Link_{\Linkh_{BA_n^m(\OK)}(w)  }(\sigma ) \ast T(\sigma)\] for each carrying simplex $\sigma$. It makes sense to replace all of these simplices simultaneously because every simplex of $\Linkh_{BA_n^m(\OK)}(w)$ contains at most one carrying subsimplex and because the subdivisions $T(\sigma)$ do not subdivide the boundary of $\sigma$. See \autoref{SubdividedSimplices}.
\begin{figure}[!ht]  
\centering
\labellist
\Large \hair 0pt
\endlabellist
\includegraphics[scale=5.5]{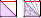} 
\caption{Examples of a carrying $1$-simplex (red) and a carrying $2$-simplex (purple) before and after subdivision. Both subdivisions can be executed simultaneously.}
\label{SubdividedSimplices}
\end{figure} 
The space $X$ is homeomorphic to $\Linkh_{BA_n^m(\OK)}(w)$ but has extra vertices which we will use to construct our retraction.

Let $\mr{vert}$ denote functor that sends a simplicial complex to its set of vertices. Then \[\mr{vert}(X)=\mr{vert}\left(\Linkh_{BA_n^m(\OK)}(w)\right) \cup \bigcup_{\sigma \, \text{carrying}} \mr{vert}(T(\sigma)). \] Note that this union is not a disjoint union as the vertices of $\gamma_\sigma$ are vertices of $\Linkh_{BA_n^m(\OK)}(w)$ and of $T(\sigma)$. Define 
\[\pi \colon \mr{vert}(X) \lra \mr{vert}\left(\Linkh{}^<_{BA_n^m(\OK)}(w) \right)\] via the formula: \[ \pi(y) =
  \begin{cases}
    y^\pi& \text{ if } y \text{ is a vertex of } \Linkh{}_{BA_n^m(\OK)}(w),
    
    \\
    H_{\sigma}(y) & \text{ if } y \text{ is a vertex of } T(\sigma).
    \end{cases} \] 
If $y$ is a vertex of both  $T(\sigma)$ and  $\Linkh{}_{BA_n^m(\OK)}(w)$, then $  H_{\sigma}(y)=y^\pi$ by construction so $\pi$ is a well-defined function on vertices. We now check that $\pi$ induces a simplicial map. Let $\tau=\{x_0,\ldots,x_p \} \subseteq X$ be a simplex. We will show that $\{\pi(x_0),\ldots,\pi(x_p)\}$ forms a simplex. We will consider the following cases:

\vspace{.5em}
\noindent {\bf Case: $\tau$ contains no interior vertices of any $T(\sigma)$.} 
\vspace{.5em} 

Since $\tau$ does not contain any vertices in the interior of any $T(\sigma)$, we can view $\tau = \{x_0,\ldots,x_p\}$ as a simplex in $\Linkh{}_{BA_n^m(\OK)}(w)$. If $\tau$ is not additive, then by the definition of $\pi$ its image $\pi(\tau)=\{x_0^\pi,\ldots,x_p^\pi\}$ is a non-additive simplex in $\Linkh{}^<_{BA_n^m(\OK)  }(w)$. 

Now assume $\tau$ is internally additive; the externally additive case is similar. Reorder and pick representatives so that $\vec x_0=\vec x_1+\vec x_2$. As in the previous paragraph, the vertices $\{x_1^\pi,\ldots,x_p^\pi\}$ span a non-additive simplex. We will now check that $\pi(\sigma)=\{x_0^\pi, x_1^\pi,\ldots,x_p^\pi\}$ forms a simplex by checking that $\{x_0^\pi, x_1^\pi,x_2^\pi\}$ is an additive simplex.
%
%
Note that $\{x_0,x_1,x_2\}$ is not carrying since all of the carrying simplices have been subdivided. Thus $\{x_0^\pi,x_1^\pi,x_2^\pi\}$ forms a simplex. Observe that \href{i}{Condition \eqref{i}} implies that the sum of the submodules $x_0^\pi,x_1^\pi,x_2^\pi$ satisfies  \[ x_0^\pi + x_1^\pi + x_2^\pi+w=x_0 \oplus x_1 \oplus w.   \]  Since the module $x_0^\pi + x_1^\pi + x_2^\pi+w$ is only rank $3$, $\{x_0^\pi,x_1^\pi,x_2^\pi\}$ must be additive. Thus,  $\pi(\sigma)=\{x_0^\pi, \ldots, x_p^\pi\}$ forms a simplex in $\Linkh{}^<_{BA_n^m(\OK)  }(w)$.

\vspace{.5em}
\noindent {\bf Case: $\tau$ contains interior vertices of some $T(\sigma)$.} 
\vspace{.5em}

Suppose that $\tau$ contains interior vertices of $T(\sigma)$ and that $\sigma=\{v_0,v_1,v_2\}$ is an internally additive carrying $2$-simplex; the externally additive case is similar. Since any simplex of $\Linkh{}_{BA_n^m(\OK)}(w)$ can contain at most one carrying subsimplex, we can decompose $\tau$ as the join $\tau = \alpha \ast \beta$ of simplices $\alpha,\beta \subset X$ such that \begin{itemize} 
\item $\alpha = T(\sigma) \cap \tau$,
\item $\beta$ contains no interior vertices of any $T(\sigma')$ for any carrying simplex $\sigma'$. 
\end{itemize}  
Note that $\pi(\alpha)=H_\sigma(\alpha)$ is a simplex of $\Link{}^<_{BA(v_0 \oplus v_1 \oplus w)(\OK)}(w)$. Since the star of any simplex in $T(\sigma)$ is contained in \[ \Link_{\Linkh_{BA_n^m(\OK)}(w)  }(\sigma ) \ast T(\sigma),\] $\beta$ is a simplex of 
\[\Link_{ \Linkh_{BA_n^m(\OK)}(w)}(v_0,v_1,v_2) = \Link_{B_n^m(\OK)}(v_0,v_1,w) = \Link_{B_n^m(\OK)}(v_0^{\pi},v_1^{\pi},w) .\] 
There is a natural inclusion \[ \Link{}^<_{BA(v_0 \oplus v_1 \oplus w)(\OK)}(w)  \ast \Link^<_{B_n^m(\OK)}(v_0^{\pi},v_1^{\pi},w) \hookrightarrow \Linkh^<_{BA_n^m(\OK)}(w)\] 
since simplices in $\Link{}^<_{BA(v_0 \oplus v_1 \oplus w)(\OK)}(w) $ all arise from (possibly augmented) partial frames for $v_0 \oplus v_1 \oplus w = v_0^{\pi} \oplus v_1^{\pi} \oplus w$, whereas simplices in $\Link^<_{B_n^m(\OK)}(v_0^{\pi},v_1^{\pi},w)$ all arise from non-augmented  partial frames for a direct complement of $v_0^{\pi} \oplus v_1^{\pi} \oplus w$ in $\OK^{n+m}$. 
In particular, $\pi(\sigma)=\pi(\alpha) \ast \pi(\beta)$ forms a simplex in $\Linkh_{BA_n^m(\OK)}(w)$.

\vspace{1em} 

Having checked that $\pi$ gives a simplicial map $X\m \Linkh{}^<_{BA_n^m(\OK)}(w)$, we use this to construct the desired retraction. Let $\Delta^d$ denote the standard $d$-simplex. Pick for each carrying simplex $\sigma$, pick a homeomorphism $h_\sigma \colon \Delta^{\mr{dim}(\sigma)}\m T(\sigma)  $ which is simplicial on the boundary. They induce homeomorphisms \[\mr{id} \ast h_\sigma \colon \mr{Star}_{\Linkh_{BA_n^m(\OK)}(w)}(\sigma ) \overset{\cong}{\lra} \Link_{\Linkh_{BA_n^m(\OK)}(w)}(\sigma ) \ast T(\sigma)\] and assemble to give a homeomorphism \[h \colon\Linkh_{BA_n^m(\OK)}(w) \lra X.\] The map $h$ is not simplicial. However, since simplices in $\Linkh{}^<_{BA_n^m(\OK)}(w)$ cannot be carrying by \href{iii}{Condition \eqref{iii}}, the composition \[\pi \circ h \circ \iota \colon \Linkh{}^<_{BA_n^m(\OK)}(w) \lra \Linkh{}^<_{BA_n^m(\OK)}(w) \] is simplicial. Thus we can check that it is the identity by checking it is the identity on vertices. This follows from \href{iii}{Condition \eqref{iii}}. Thus \[\pi \circ h \colon \Linkh{}_{BA_n^m(\OK)}(w) \lra \Linkh{}^<_{BA_n^m(\OK)}(w) \] is a retraction of $\iota \colon \Linkh{}^<_{BA_n^m(\OK)}(w) \lra \Linkh{}_{BA_n^m(\OK)}(w)$.\end{proof}

\autoref{RETLEM} has the following corollary. 

\begin{corollary}\label{4.17} Let $\OK$ denote the Gaussian integers or Eisenstein integers. Let $w$ be a line in $BA^n_m(\OK)$ with $F(w) \neq 0$.  If $\Linkh_{BA_n^m(\OK)}(w)$ is $d$-connected, so is $\Linkh{}^<_{BA_n^m(\OK)}(w)$.
\end{corollary}

 The following result is a direct adaptation of Church--Putman \cite[Proposition 4.14]{CP}. 

\begin{lemma} \label{LinkInduction} Let $\OO$ be a Euclidean domain.  Let $n \geq 1$ and $m \geq 0$ such that $m+n \geq 2$. Assume that $BA^{m'}_{n'}(\OO)$ is Cohen--Macaulay of dimension $n'$ for all $1 \leq n' < n$ and all $m' \geq 0$ satisfying $m'+n' = m+n$. Then for every $p$-simplex $\sigma$ of $BA^{m}_{n}(\OO)$, the link $\Link_{BA^{m}_{n}(\OO)}(\sigma)$ is Cohen--Macaulay of dimension $(n-p-1)$. 
\end{lemma}

\begin{proof} The proof of Church--Putman \cite[Proposition 4.14]{CP} applies without modification; we summarize it briefly. If a $p$-simplex $\sigma$ is additive, then $\Link_{BA^{m}_{n}(\OO)}(\sigma) \cong B^{m+p}_{n-p}(\OO)$ is Cohen--Macaulay of dimension $(n-p-1)$ by \autoref{thm:maazen}.

Next, suppose we have a non-additive $(n-1)$-simplex $\sigma = \{v_1, \ldots, v_n\}$. Then $\Link_{BA^{m}_{n}(\OO)}(\sigma)$ contains the vertex corresponding to $\vec v_0 = \vec e_1 + \vec v_1$ (if $m \geq 1$) or $\vec v_0 = \vec v_1 + \vec v_2$ (if $m=0$), so $\Link_{BA^{m}_{n}(\OO)}(\sigma)$ is non-empty. 

Finally, suppose we have a non-additive $p$-simplex $\sigma$ with $p<n-1$. Then $\Linkh_{BA^{m}_{n}(\OO)}(\sigma) \cong BA_{n-p-1}^{m+p+1}(\OO)$ is Cohen--Macaulay of dimension $(n-p-1)$ by assumption. Each vertex \[v \in \Link_{BA^{m}_{n}(\OO)}(\sigma)  \setminus \Linkh_{BA^{m}_{n}(\OO)}(\sigma)\] has $\Link_{\Link_{BA^{m}_{n}(\OO)}(\sigma)  } (v) \cong B^{m+p+1}_{n-p-1}(\OO)$ contained in  $\Linkh_{BA^{m}_{n}(\OO)}(\sigma)$, and so the addition of each such vertex $v$ has the effect of coning off a subcomplex of $\Linkh_{BA^{m}_{n}(\OO)}(\sigma)$ that is Cohen--Macaulay of dimension $(n-p-2)$. The result follows by Church--Putman \cite[Lemma 4.13]{CP}. 
\end{proof}

We now prove \autoref{BAnOtherRings}, which states that $BA_n^m(\OK)$ is Cohen--Macaulay of dimension $n$ for $\OK$ the Gaussian integers or Eisenstein integers. Other than the proof of \cite[Proposition 4.17]{CP}, the proof of Church--Putman \cite[Theorem C']{CP} goes through without modification for all rings which are additively generated by multiplicative units and have a multiplicative Euclidean function with the property that if $|a|=|b|>0$, then there is a unit $u$ with $|a-ub|<|a|$. In place of \cite[Proposition 4.17]{CP}, we instead invoke our \autoref{4.17}. That the ring is Euclidean is used throughout the proof \cite[Theorem C']{CP}; that it is additively generated by multiplicative units is used in the base case of the induction \cite[Page 21]{CP} and appears here as \autoref{linkUnits}.



Recall that a combinatorial $i$-sphere is a simplicial complex that is PL-homeomorphic to the boundary of an $(i+1)$-simplex and a combinatorial $i$-disk is a simplicial complex that is PL-homeomorphic to an $i$-simplex. Moreover, links of $p$-simplices in the interior of combinatorial $i$-spheres and combinatorial $i$-disks are combinatorial $(i-p-1)$-spheres.
Given a simplicial complex $X$, the simplicial approximation theorem implies that we can represent every homotopy class of maps $S^i \to X$ by a simplicial map from a combinatorial $i$-sphere to $X$. It also implies that every null-homotopic simplicial map from a combinatorial $i$-sphere to $X$ can be extended to a simplicial map from a combinatorial $(i+1)$-disk to $X$. Replacing the star of a $p$-simplex in a combinatorial $i$-sphere by a different $(i-p)$-disk (with the same combinatorial $(i-p-1)$-sphere as boundary) results in a combinatorial $i$-sphere again. We will use this fact while construct homotopies in the following proof. For a detailed introduction to the topic we refer the reader to Rourke-Sanderson \cite{Rourke-Sanderson}.


\begin{proof}[Proof of  \autoref{BAnOtherRings}] We summarize Church--Putman \cite[Proof of Theorem C']{CP}. The proof proceeds by induction on $n$ and $m$. The base case is that $BA_1^m(\OK)$ is connected for all $m \geq1$ which was proven in \autoref{linkUnits}. Now let $n \geq 1$ and $m \geq 0$ such that $m+n \geq 2$, and assume that $BA^{m'}_{n'}(\OK)$ is Cohen--Macaulay of dimension $n'$ for all $1 \leq n' < n$ with $2 \leq n'+m' \leq n+m$. 

The links of simplices in $BA^{m}_{n}(\OK)$ are appropriately highly-connected by \autoref{LinkInduction}, so it suffices to prove that $BA^{m}_{n}(\OK)$ is $(n-1)$-connected. 

Let $\phi \colon S^i \to BA^{m}_{n}(\OK)$ be a simplicial map from a combinatorial $i$-sphere for $i\le n-1$. Let 
\[M(\phi) \coloneqq \max_{ \text{vertices } x \in S^i} F( \phi(x))\]
where $F$ is defined in \autoref{Defn:fF}; the function $M$ quantifies the `badness' of the map $\phi$. Our goal is to homotope $\phi$ to reduce $M$. Then we can inductively homotope the map $\phi$ to a map $\phi'$ for which $\phi'(x)$ has $(m+n)^{th}$ coordinate zero for every vertex $x \in S^i$. The image of $\phi'$ is in the star of the vertex $e_{m+n}$, and so it can then be homotoped to the constant map at $e_{m+n}$. 
 
Assume $M(\phi)=M>0$. Following Church--Putman, we proceed in four steps. In the first step, we homotope $\phi$ so that for every simplex $\sigma \in S^i$ mapping to an additive simplex of $BA^{m}_{n}(\OK)$ satisfies $F(\phi(x)) < M$ for all vertices $x \in \sigma$. We must achieve this homotopy without increasing $M(\phi)$. Choose $\sigma$ of maximal dimension $q$ among those simplices in $S^i$ satisfying the following properties $(\ast)$: 
\begin{itemize}
\item $\phi(\sigma) = \{v_0, \ldots, v_p\}$ is additive, say, $\vec v_0 = \vec v_1 + \vec v_2$ for some generators of $v_0, v_1, v_2$, 
\item $F(\phi(x))=M$ for some vertex $x \in \sigma$, and
\item $F(\phi(x))=M$ for every vertex $x \in \sigma$ with $\phi(x) \in \{v_3, \ldots,v_p\}$. 
\end{itemize} 
Note that $q\ge p$. By assumption of maximality, 
\[\phi(\Link_{S^i}(\sigma)) \; \subseteq \;  \Link^{<}_{BA^{m}_{n}(\OK)}(\{v_0, v_1, \ldots, v_p\}) \; = \;  \Link^{<}_{B^{m}_{n}(\OK)}(\{v_1, \ldots, v_p\}).\]
For the last equality we have used the assumption that, possibly after re-indexing $v_0, v_1, v_2$,  $F(v_i)=M$ for some $i=1, \ldots,p$.   The complex $ \Link^{<}_{B^{m}_{n}(\OK)}(\{v_1, \ldots, v_p\})$ is  $(n-p-2)$-connected by  \autoref{LinkB}. Thus, the restriction 
\[S^{i-q-1} \cong \Link_{S^i}(\sigma) \lra  \Link^{<}_{B^{m}_{n}(\OK)}(\{v_1, \ldots, v_p\})  \]
 is null-homotopic as $i-q-1 \le n-p-2$. This implies 
 there is a combinatorial $(i-q)$-disk $D$ whose boundary is isomorphic to the combinatorial $(i-q-1)$-sphere $ \Link_{S^i}(\sigma)$ and a map 
 \[g \colon D \lra  \Link^{<}_{B^{m}_{n}(\OK)}(\{v_1, \ldots, v_p\})  \]
 extending  $\phi\vert_{\Link_{S^i}(\sigma)}$. Because $g$ maps to the link of $\phi(\sigma)$, we may define the join of the maps
 \[ ( \phi|_{\sigma} \ast g ) \colon  (\sigma \ast D) \to BA^m_n.\] 
 Let $Z$ be 
the combinatorial $i$-sphere
 given by replacing $\mr{Star}_{S^i}(\sigma)$ in $S^i$ with 
$D \ast \partial \sigma$. 
 Let $\hat \phi \colon Z \m BA_n^m(\OK)$ be given by the formula
\[\hat \phi(y) =
  \begin{cases}
    \phi(y) & \text{ if } y \in Z \setminus (\partial \sigma \ast D),\\ 
 (\phi|_{\sigma}\ast g)(y) & \text{ if } y \in  \partial \sigma \ast D. 
\end{cases} \]
Note that this map is continuous as $\phi$ and $  (\phi|_{\sigma}\ast g)$ coincide on $\partial \sigma \ast \partial D \cong \partial \mr{Star}_{S^i}(\sigma)$. Moreover, observe that $\phi$ and $\hat \phi$ are homotopic through $\phi|_{\sigma} \ast g$ (extended by the constant homotopy outside of the star of $\sigma$).
The new map $\hat \phi$ has one fewer maximal simplices satisfying $(\ast)$. A similar argument applies to externally additive simplices. Iterating this procedure produces the desired map. 

In the second step, Church--Putman homotope $\phi$ so that if vertices $x_1, x_2 \in S^i$ satisfy $\phi(x_1)=\phi(x_2)=v$ with $F(v)=M$, then $x_1, x_2$ are not joined by an edge. This new map must not increase $M(\phi)$ and must retain the properties achieved in Step 1. Choose a simplex $\sigma$ of maximal dimension with the properties that $\phi|_{\sigma}$ is not injective and $F(\phi(x))=M$ for every vertex $x \in \sigma$. Again let $p=\dim(\phi(\sigma))$. Then, using the properties achieved in Step 1, $\Link_{S^i}(\sigma)$ must map to the subcomplex
$$\Link^{<}_{B^{m}_{n}(\OK)}(\phi(\sigma))\; \subseteq \;  \Link^{<}_{BA^{m}_{n}(\OK)}(\phi(\sigma)). $$ This subcomplex is $(n-(p-1)-2)$--connected by  \autoref{LinkB}, so again we can homotope $\phi$ to remove the simplex $\phi(\sigma)$ while preserving our desired properties. 

In the third step, Church--Putman further homotope the map $\phi$ so that it retains the properties from Steps 1 and 2, and has the additional property that, whenever vertices $x_1, x_2 \in S^i$ satisfy $$F(\phi(x_1)) = F(\phi(x_2))=M,$$ then $x_1$ and $x_2$ are not connected by an edge.  Suppose $\{x_1, x_2\}$ is an edge violating this condition, with $\phi(x_1) = v_1$ and $\phi(x_2) = v_2$. Pick representatives $\vec v_1$ and $\vec v_2$. By \autoref{lem0}, there is a unit $u$ with $F(\vec v_1-u \vec v_2)<F(\vec v_1)$. Let $v_0 = \mr{span}(\vec v_1-u \vec v_2)$. By the property ensured in Step 2, $v_1 \neq v_2$. Thus $\vec v_1-u \vec v_2 \neq \vec 0$ and so $v_0$ is a line. Given the property ensured in Step 1,  the image of $\Link_{S^i}(\{x_1, x_2\})$ is contained in $\Link_{BA_n^m(\OK)}(\{v_0,v_1, v_2\})$. We can therefore homotope the map $\phi$ to map $\{x_1,x_2\}$ to the concatenation of $\{v_1,v_0\}$ with $\{v_0,v_2\}$. In particular, we replace the simplicial structure on $$ \mathrm{Star}_{S^i}(\{x_1, x_2\}) = \{x_1, x_2\} \ast \mathrm{Link}_{S^i}(\{x_1, x_2\})$$ with the join of the barycentric subdivision of the edge $\{x_1, x_2\}$ and $\mathrm{Link}_{S^i}(\{x_1, x_2\})$. 
This procedure removes the edge $\{v_1, v_2\}$ from the image of the map and preserves the properties from the previous steps. 

In the final step, Church--Putman homotope the map $\phi$ to reduce $M(\phi)$. Let $x \in S^i$ be a vertex such that $\phi(x)=v$ with $F(v)=M$. The properties established in the previous steps ensure that
\[\phi(\Link_{S^i}(x)) \; \subseteq \;  \Linkh{}^<_{BA^{m}_{n}(\OK)}(v).\] The complex $\Linkh_{BA^{m}_{n}(\OK)}(v) \cong BA_{n-1}^{m+1}(\OK)$ is $(n-2)$--connected by inductive hypothesis, thus so is $\Linkh{}^<_{BA^{m}_{n}(\OK)}(v)$  by \autoref{4.17}. We can therefore homotope $\phi$ to remove $v$ from its image while preserving the properties from previous steps. Iterating this final step will reduce $M(\phi)$ and complete the proof. 
\end{proof}

\subsection{New non-connectivity results}

In this subsection, we show that $BA_n(\OO)$ may not always be Cohen--Macaulay, but is always highly connected. We begin with a general lemma about links in simplicial complexes.

\begin{lemma} \label{linkGeneral}
Let $X$ be a simplicial complex and fix a simplicial structure on $S^1$. Let $x$ be a vertex of $S^1$ and let $y,z \in \Link_{S^1}(x)$. Let $\phi \colon S^1 \m X$ be a simplicial map such that no vertex other than $x$ maps to $\phi(x)$. If $\phi_*([S^1]) = 0$ in $H_1(X)$, then $[\phi(y)]=[\phi(z)]$ in $\pi_0(\Link_X(\phi(x)))$.
\end{lemma}

\begin{proof}
Let $C_*$ denote cellular chains. Suppose that  $\phi_*([S^1]) = 0$, and let $\alpha \in C_2(X)$ be a chain such that \[\partial(\alpha)=\phi_*([S^1]) \in C_1(X).\] 
The chain $\alpha$ can be written as $\alpha_1+\alpha_2$, with $\alpha_1$ a sum of 2-simplices that have $\phi(x)$ as a vertex and $\alpha_2$ a sum of 2-simplices that do not have $\phi(x)$ as a vertex. Let $\beta \in C_1(X)$ denote the chain associated to the simplicial path with vertices $\phi(y)$, $\phi(x)$, and $\phi(z)$. An instance of this complex is shown in \autoref{LoopLink}.
{\begin{figure}[!ht]  
\centering
\labellist
\Large \hair 0pt
\pinlabel { \color{violet} ${ \phi(y) }$} [tr] at -1 23
\pinlabel { \color{violet} $\small{ \phi(x) }$} [r] at 10 32
\pinlabel { \color{violet} $\small{ \phi(z) }$} [r] at 43 32
\pinlabel { \color{violet} $\small{ \beta }$} [r] at 25 36
\pinlabel { \color{teal} ${ \alpha_1 }$} [tl] at 14 20
\pinlabel { \color{black} $\small{ \phi_*([S^1]) }$} [l] at  41 16

\endlabellist \quad\\ \quad \\ 
\includegraphics[scale=2.5]{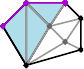} 
\caption{An illustrative example of the chain $\alpha$.} 
\label{LoopLink}
\end{figure}
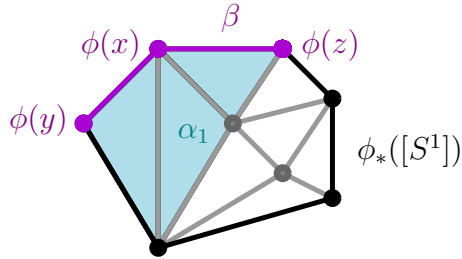 
}

Each $1$-simplex in the boundary of a 2-simplex appearing in $\alpha_1$ either contains the vertex $\phi(x)$, or is contained in $\Link_{X}(\phi(x))$. By assumption on $\phi$, the boundary $\partial(\alpha_1)$ must only pass through the vertex $\phi(x)$ once, and so it must be a sum of $\beta$ and terms in the link of $\phi(x)$. Hence \[\partial(\alpha_1)=\beta \in C_1(X,\Link_{X}(\phi(x))).\]

Thus, $[\beta]$ vanishes in $H_1(X,\Link_{X}(\phi(x)))$. This implies that its image $\delta([\beta])$ vanishes  in $H_0(\Link_{X}(\phi(x)))$ under the connecting homomorphism $\delta$ in the long exact sequence of the pair $(X,\Link_{X}(\phi(x)))$. Since 
\[\delta([\beta]) = [\phi(y)] -[\phi(z)],\]
this implies that $\phi(y)$ and $\phi(z)$ are in the same path component.
\end{proof}

\begin{definition}
We say that a ring $\OO$ \emph{has detours} if there are $r_1, r_2 \in \OO$ such that \begin{enumerate}[\noindent (i)]
\item $r_1 - r_2$ is not a sum of units.
\item There is a simplicial path in $B_2(\OO)$ from $\mr{span}\begin{psmallmatrix}r_1 \\ 1\end{psmallmatrix}$ to $\mr{span}\begin{psmallmatrix}r_2 \\ 1\end{psmallmatrix}$ that avoids $\mr{span}\begin{psmallmatrix}1 \\ 0\end{psmallmatrix}$.
\end{enumerate}
\end{definition}

An example of a detour is given in \autoref{sqrt7loop} for $\cO_K = \bZ[\sqrt{7}]$. Each vertex is labelled by a vector spanning the corresponding line. 

\begin{figure}[t]  
	\centering
	\labellist
	\Large \hair 0pt
	\pinlabel { \color{black} $\small{  \displaystyle \begin{pmatrix} 1 \\ 0 \end{pmatrix} }$} [tr] at -1 26
	\pinlabel { \color{black} $\small{  \displaystyle \begin{pmatrix} \sqrt{7} \\ 1 \end{pmatrix} }$} [tl] at 25 26
	\pinlabel { \color{black} $\small{ \displaystyle \begin{pmatrix} 8 \\-3 \end{pmatrix} }$} [bl] at  25 -1
	\pinlabel { \color{black} $\small{  \displaystyle \begin{pmatrix} -3 \\1 \end{pmatrix} }$} [br] at  -1 -1
	\endlabellist
	\includegraphics[scale=2.5]{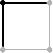} 
	\caption{A loop in $BA_2(\Z[\sqrt 7])$ coming from a detour (marked in gray).} 
	\label{sqrt7loop}
\end{figure}
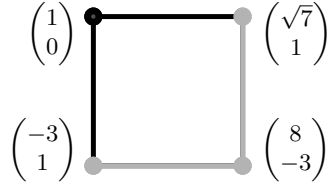 
\begin{proposition} \label{H1}
If $\OO$ has detours, then $H_1(BA_2(\OO)) \neq 0$.
\end{proposition}

\begin{proof}

Let $v_1=\mr{span}\begin{psmallmatrix}r_1\\1\end{psmallmatrix}$  and $v_2=\mr{span}\begin{psmallmatrix}r_2\\1\end{psmallmatrix}$. We saw in the proof of \autoref{linkUnits} that two vertices  $\mr{span}(x_1,\ldots,x_m,1)$ and $\mr{span}(y_1,\ldots,y_m,1)$ are in the same path component of $BA_1^m(\OO)$ if and only if $x_i-y_i$ is a sum of units for each $i$. It follows that $v_1$ and $v_2$ are not in the same path component of \[BA_1^1(\OO) = \Linkh_{BA_2(\OO)}\left( e_1 \right)=\Link_{BA_2(\OO)}\left(e_1 \right).\] Consider the loop which is a concatenation of a detour from $v_1$ to $v_2$ with the path given by the three vertices $v_1$, $e_1$, and $v_2$. This loop is not zero in $H_1(BA_2(\OO))$ by \autoref{linkGeneral}. 
\end{proof}

Before we give several examples of rings with detours, we need the following lemma.

\begin{lemma} \label{pathBZ}
The full subcomplex of $B_2(\Z)$ minus the vertex $e_1$ is connected.

\end{lemma}

\begin{proof}
Let $v_1,v_2 \neq e_1$ be vertices in $B_2(\Z)$. Since $B_2(\Z)$ is connected \cite[Theorem 4.2]{CP}, we can find a simplicial path $\gamma$ from $v_1$ to $v_2$. Suppose the path contains $e_1$. By removing loops, we may assume that $\gamma$ only passes through $e_1$ once. Let $w_1$ and $w_2$ be the vertices adjacent to $e_1$ in the path. Since $\Link_{BA_2(\Z)}(e_1)$ is connected \cite[Theorem C]{CP}, we can find a path $\gamma'$ from $w_1$ to $w_2$ in the link. Note that $\Link_{BA_2(\Z)}(e_1) \subset B_2(\Z)$.  Let $\gamma''$ be $\gamma$ with $\{w_1,e_1\} \cup \{e_1,w_2\}$ replaced with $\gamma'$. Observe that $\gamma''$ gives a path from $v_1$ to $v_2$ that avoids $e_1$.
\end{proof}

We now show that Euclidean quadratic number rings not generated by units have detours.

\begin{proposition} \label{exotic}
Let $\OK$ be the ring of integers in $\KK=\Q(\sqrt d)$ for $d$ squarefree. Assume that $\OK$ is not generated by units but is Euclidean. Then $\OK$ has detours. 
\end{proposition}


\begin{proof}

Let $\delta=\sqrt d$ for $d \neq 1$ (mod $4$) and $\delta=\frac{1+\sqrt d}{2}$ for $d = 1$ (mod $4$) so that $\OK=\Z[\delta]$. Since $\OK$ is not generated by units, $\delta$ is not a sum of units. Thus, it suffices to find a path from $\begin{psmallmatrix} \delta \\ 1\end{psmallmatrix}$ to $\begin{psmallmatrix} 0 \\ 1\end{psmallmatrix}$ that avoids $\begin{psmallmatrix} 1 \\ 0\end{psmallmatrix}$. 

\vspace{.5em}

We first consider the case $d>0$.  By Dirichlet's unit theorem, $\OK$ has infinitely many units so there is a unit of the form $a+b \delta$ with $b \neq 0$, $a,b \in \Z$. Note that the lines spanned by the vectors \[\left\{\begin{pmatrix} \delta \\ 1\end{pmatrix}, \begin{pmatrix} a \\ -b\end{pmatrix}  \right\} \] form an edge in $B_2(\OK)$. By \autoref{pathBZ}, there is a path in $B_2(\Z) \subseteq B_2(\OK)$ from $\begin{psmallmatrix} a \\ -b\end{psmallmatrix}$  to $\begin{psmallmatrix} 0 \\ 1\end{psmallmatrix}$ that avoids $\begin{psmallmatrix} 1 \\ 0\end{psmallmatrix}$. The concatenation of this path with the previous edge is a detour. 


 

\vspace{.5em}
 
Now assume $d<0$. The only Euclidean quadratic imaginary number rings have $d=-1,-2,-3,-7$ and $-11$. For $d=-1$ and $d=-3$, these are generated by units while the other three rings are not generated by units (see e.g. Ashrafi--V\'amos \cite[Theorem 7]{AshrafiVamos}). Note that the units in the case $d=-2,-7,$ or $-11$ are just $\pm 1$ and so $(a+b\delta)-(c+d\delta)$ is not a sum of units whenever $b \neq d$. Unlike for real quadratic number rings where we had a conceptual construction of detours, in the imaginary case, will just exhibit an explicit detour for each ring. 
 
\vspace{.5em}
\noindent {\bf Case: $d=-2$.} 
\vspace{.5em} 
The path with vertices spanned by the following vectors is a detour:
 \[ \begin{pmatrix} \delta \\ 1\end{pmatrix},  \begin{pmatrix} 1 \\ -  \delta \end{pmatrix} ,  \begin{pmatrix} 0 \\ 1 \end{pmatrix}.\]
 
 \vspace{.5em}
\noindent {\bf Case: $d=-7$.} 
\vspace{.5em} 
The path with vertices spanned by the following vectors is a detour:
 \[ \begin{pmatrix} \delta \\ 1\end{pmatrix},  \begin{pmatrix} 3-\delta \\ -  \delta \end{pmatrix} ,  \begin{pmatrix} -1+2\delta \\ 1 \end{pmatrix}.\]
 

 \vspace{.5em}
\noindent {\bf Case: $d=-11$.} 
\vspace{.5em} 
The path with vertices spanned by the following vectors is a detour:
 \[ \begin{pmatrix} \delta \\ 1\end{pmatrix},  \begin{pmatrix} 2 \\  1-\delta \end{pmatrix} , \begin{pmatrix} \delta \\  2 \end{pmatrix}, \begin{pmatrix} 1 \\  1-\delta \end{pmatrix}, \begin{pmatrix} 0 \\  1 \end{pmatrix}. \qedhere\] \end{proof}
 


%
%


\begin{remark}
\label{Listd}

The quadratic norm-Euclidean number rings have been completely classified. They are the rings of integers $\OK$ of $\KK=\Q(\sqrt d)$ with \[d\in \{-11,-7,-3,-2,-1, 2, 3, 5, 6, 7, 11, 13, 17, 19, 21, 29, 33, 37, 41, 57, 73\};\] see e.g. Stark \cite[Theorem 8.21]{Stark}. Ashrafi--V\'amos \cite[Theorem 7]{AshrafiVamos} completely characterized which quadratic number rings are generated by units. When $d\not\equiv 1 \pmod 4$ and $d>0$, $\OK$ is generated by units if and only if $d$ can be written as $d = a^2\pm 1$ for some $a\in \Z$. When $d\equiv 1 \pmod 4$ and $d>0$, the ring is generated by units if and only if $d$ can be written as $d=a^2 \pm 4$ for some $a\in \Z$. For $d<0$, the ring is generated by units if and only if $d \in \{-3,-1\}$. Thus, the norm-Euclidean number rings that are not generated by units are the rings of integers $\OK$ of $\KK=\Q(\sqrt d)$ with \[d\in  \{-11,-7,-2,6,7,11,17,19,33,37,41,57,73\}.\] On the other hand, there are Euclidean (but not norm-Euclidean) quadratic number rings that are not generated by units, such as the ring of integers in $\Q( \sqrt{69})$. Our results apply equally well to these rings. 

\end{remark}

We have just shown that it is not true that $BA_n(\OO)$ is spherical for all Euclidean domain. However, it is always highly connected.

\begin{proposition} \label{n2conn}
Let $\OO$ be a Euclidean domain. Then $BA_n^m(\OO)$ is $(n-2)$-connected.
\end{proposition}

\begin{proof}
Since \autoref{thm:maazen} says that $B_n^m(\OO)$ is $(n-2)$-connected, it suffices to show that $B_n^m(\OO) \m BA_n^m(\OO)$ induces a surjection on free homotopy classes $[S^i,-]$ for $ i\leq n-2$. We will in fact show that it is a surjection for $i \leq n-1$. Fix $i \leq n-1$ and let $\phi \colon S^i \m BA_n^m(\OO)$ be a simplicial map with respect to a combinatorial triangulation of $S^i$. Our goal is to show that $\phi$ is homotopic to a map to $B_n^m(\OO)$. Suppose that $\{v_0,v_1,v_2\}$ is an internally additive simplex in the image of $\phi$. Then the link
\[\Link_{BA^m_n(\OO)}(\{v_0,v_1,v_2\}) \cong \Link_{B_n^m(\OO)}(\{v_1,v_2\}).\]
is $(n-4)$-connected by \autoref{thm:maazen}. 
As in the proof of \cite[Lemma 2.4]{MPP} and the first step of the proof of \autoref{BAnOtherRings}, we can homotope the map $\phi$ to avoid the simplex $\{v_0,v_1,v_2\}$ without introducing new additive simplices to its image. Iterating this procedure, and the analogous procedure for externally additive simplices, produces a map to $B_n^m(\OO)$.
\end{proof}

\section{Presentations of Steinberg modules and vanishing of cohomology} 

In this section, we use our (non-)connectivity results to deduce the main theorems of the paper. We begin with a review of a useful tool: the spectral sequence associated to a map of posets, originally due to Quillen \cite{Quillen-Poset}. 

\subsection{The map of posets spectral sequence}
\label{Posection}
Let $\mathbf{Y}$ be a poset. Associated to $\mathbf{Y}$ is the simplicial complex $\Delta(\mathbf{Y})$ of non-degenerate simplices in the nerve of $\mathbf{Y}$. A $p$-simplex of $\Delta(\mathbf{Y})$ corresponds to a $(p+1)$-chain $y_0 < y_1 < \cdots < y_p$ of elements in $\mathbf{Y}$.  
The \emph{dimension} of $\mathbf{Y}$ is defined to be the dimension of $\Delta(\mathbf{Y})$, and we let $|\mathbf{Y}|$ denote the geometric realization of $\Delta(\mathbf{Y})$. We note that, if $Y$ is a simplicial complex and $\mathbf{Y}$ the corresponding poset of simplices under inclusion, then  $\Delta(\mathbf{Y})$ is the barycentric subdivision of $Y$, and there is a homeomorphism $Y \cong |\mathbf{Y}|$.

For an element $y \in \mathbf{Y}$, recall we defined the subposets,
\[\mathbf{Y}_{\leq y} \coloneqq \{ y' \in \mathbf{Y} \; | \; y' \leq y \} \qquad \text{ and } \qquad \mathbf{Y}_{> y} \coloneqq \{ y' \in \mathbf{Y} \; | \; y' > y \}.\]
 
  \begin{definition}  Let $\mathbf{Y}$ be a poset. Let $T$ be a functor from $\mathbf{Y}$ (viewed as a category) to the category $\underline{\mr{Ab}}$  of abelian groups. Define chain groups
  \[C_p(\mathbf{Y}; T) \coloneqq \bigoplus_{y_0<\cdots<y_p \in \mathbf{Y}} T(y_0)\]
with a differential $\sum_{i=0}^p (-1)^i d_i$, with the face maps $d_i$ given by
  \begin{align*}   d_i\colon \bigoplus_{y_0<\cdots<y_p } T(y_0) &\longrightarrow \bigoplus_{y_0<\cdots < \hat y_i< \cdots <y_p } T(y_0)  \qquad \qquad (1 \leq i \leq p) \\
   d_0\colon \bigoplus_{y_0<\cdots<y_p } T(y_0) &\longrightarrow \bigoplus_{y_1< \cdots <y_p } T(y_1), 
  \end{align*}
defined as follows. For $i\neq 0$, the map $d_i$ maps the summand indexed by $(y_0<\cdots<y_p)$ to the summand indexed by $(y_0<\cdots < \hat y_i< \cdots <y_p)$, and acts by the identity on the group $T(y_0)$. The map $d_0$ maps the summand indexed by $(y_0<\cdots<y_p)$ to the summand indexed by $(y_1< \cdots <y_p)$, and the map of abelian groups $T(y_0) \to T(y_1)$ is defined by applying $T$ to the morphism $y_0<y_1$  in $\mathbf{Y}$.    \end{definition}
   
If  $T=\Z$ is the constant functor with identity maps, then $H_*(\mathbf{Y}; \Z)$ is isomorphic to the usual homology $H_*(|\mathbf{Y}|)$. The following lemma is adapted from Charney \cite[Lemma 1.3]{Charney-Generalization}. See also \cite[Lemma 3.2]{MPWY}. 
Recall that the \emph{height} of  $y \in Y$ is by definition $\dim(\Delta(\mathbf{Y}_{\leq y}))$. 

   
\begin{lemma} \label{ChainsSingleSupport}    
Suppose that $T \colon \mathbf{Y} \to \underline{\mr{Ab}}$ is a functor that is nonzero only on elements of height $m$. Then 
   \[H_p(\mathbf{Y};T) = \bigoplus_{\mr{height}(y_0)=m} \widetilde{H}_{p-1}(|\mathbf{Y}_{>y_0}|; T(y_0)).\]
\end{lemma}

\begin{definition}  Let $f\colon \mathbf{X} \to \mathbf{Y}$ be a map of posets.  For $y \in \mathbf{Y}$, define $f\backslash y$ to be the subposet of $\mathbf{X}$ 	
	\[f\backslash y\coloneqq \{ x \in \mathbf{X} \; | \; f(x) \leq y \}.\]
\end{definition}

Consider a map of posets  $f\colon \mathbf{X} \to \mathbf{Y}$, and fix a degree $q \in \Z_{\geq 0}$. Then there is a functor from the poset  $\mathbf{Y}$ to $ \underline{\mr{Ab}}$ that takes an object $y \in \mathbf{Y}$ to the abelian group $H_q(f \backslash y)$. With this functor, we may state the following theorem. The spectral sequence associated to a map $f\colon \mathbf{X} \to \mathbf{Y}$ of posets was introduced by Quillen \cite[Section 7]{Quillen-Poset}; see also Charney \cite[Section 1]{Charney-Generalization}.

\begin{theorem}[Quillen \cite{Quillen-Poset}] \label{MPSS} Let $f\colon \mathbf{X} \to \mathbf{Y}$ be a map of posets. There is a strongly convergent spectral sequence
	\[ E^2_{p,q} = H_p\Big(\mathbf{Y};  [y \mapsto H_q(f \backslash y)]\Big) \implies H_{p+q}(\mathbf{X}).  \]
\end{theorem}



\subsection{Generalized Bykovski\u\i \, presentations for the Gaussian integers and Eisenstein integers}
In this subsection, we let $\OK$ denote Gaussian integers or the Eisenstein integers and $\KK$ its field of fractions. Our objective is to  prove \autoref{Vanishing} and \autoref{BforGandE}. Recall that \autoref{BforGandE} is the statement that $\mr{Byk}_n(\OK) \m \St_n(\KK)$ is an isomorphism for all $n$. 
We can deduce  \autoref{BforGandE}  from \autoref{BAnOtherRings} using the same arguments that Church--Putman use to deduce \cite[Theorem B]{CP} from \cite[Theorem C]{CP}. We recall these arguments in the three lemmas below and the proof of \autoref{BforGandE}. 

We make the following definition, as in Church--Putman \cite[Proof of Theorem B]{CP}. 

\begin{definition} \label{BA'} For a Euclidean ring $\OO$, we let $BA_n(\OO)'$ denote the subcomplex of $BA_n(\OO)$ consisting of simplices $\{v_0, v_1, \ldots, v_p\}$ with $v_0 + v_1 + \cdots + v_p \subsetneq \OO^n$.
\end{definition}

Let $\mr{sd}$ denote the barycentric subdivision. This subcomplex is defined to give a map 
\begin{align*} \mr{span} \colon \mr{sd}(BA_n(\OO)') &\longrightarrow \cT_n(\KK) \\
\{v_0,v_1, \ldots, v_p\} & \longmapsto \KK v_0 + \KK v_1 + \cdots + \KK v_p
\end{align*}
This arises from a map of posets with domain $\cat{simp}(BA_n(\OO)')$ and target the poset defining the Tits building.

\begin{lemma}[Following  {\cite[Theorem B Step 3]{CP}}] \label{MPSSisos}  Let $\OK$ be the Gaussian integers or Eisenstein integers.  The map $\mr{span} \colon \mr{sd}(BA_n(\OK)') \to \cT_n(\KK)$ induces an isomorphism of $\bZ[\GL_n(\OK)]$-modules
$\widetilde H_{n-2}(BA_n(\OK)') \xrightarrow{\cong} \widetilde H_{n-2}(\cT_n(\KK))$.
\end{lemma} 

We could prove this by quoting Church--Putman \cite[Proposition 2.3]{CP}, but will instead prove it using the spectral sequence of \autoref{MPSS}, as a warm-up for our proof of \autoref{NoB}. 

\begin{proof}
When $n=1$, both $BA_n(\OK)'$ and $\cT_n(\KK)$ are empty, so we may assume $n\geq 2$. We consider the spectral sequence of \autoref{MPSS} associated to the functor $\mr{span}$. Observe that, given a proper nonzero subspace $V \subsetneq \KK^n$, 
\[\mr{span}\backslash V  = \Big\{ \{v_0, v_1, \ldots, v_p\} \in \cat{simp}(BA_n(\OK)') \; \Big| \; \KK v_0 + \KK v_1 + \cdots + \KK v_p \subseteq V\Big\} \cong BA(V\cap \OK^n).\]
By \autoref{BAnOtherRings} the complex $BA(V\cap \OK^n)$ is Cohen--Macaulay of dimension $\dim(V)$ (for $\dim(V) \geq 2$) or dimension $0$ (when $\dim(V)=1$), so $H_q(\mr{span}\backslash V ) = 0$ except possibly when $q=0$ or  $q=\dim(V)$. We can identify $\cT_n(\KK)_{>V} \cong \cT(\KK^n/V)$. Thus for $q>0$, we find by \autoref{ChainsSingleSupport} that
\begin{align*} E^2_{p,q} &\cong H_p\Big(\cT_n(\KK); [V \mapsto H_q(\mr{span}\backslash V )]\Big) \\ &\cong \bigoplus_{V \subseteq \KK^n, \, \dim(V)=q} \widetilde{H}_{p-1}\Big( \cT(\KK^n/V); H_{\dim(V)}(BA(V\cap \OK^n))\Big).
\end{align*}
The building $\cT(\KK^n/V)$ is spherical of dimension $\dim(\KK^n/V)-2$, and so for $q>0$, we conclude that $E^2_{p,q}$ vanishes unless $p-1 = n - \dim(V)-2$, equivalently, unless $p+q = n-1$. When $q=0$,
\[E^2_{p,q} \cong H_p(\cT_n(\KK))=0   \qquad \text{ except when $p=0$ or  $p=n-2$}.\]

 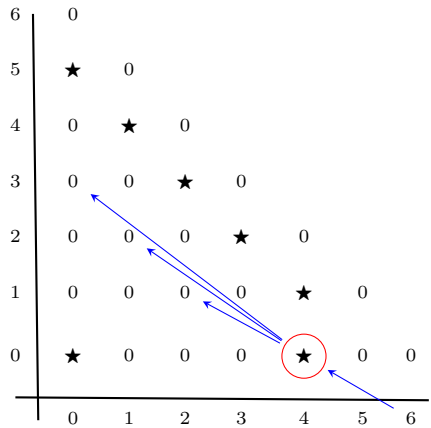
\begin{figure}[h!]    \hspace{-0cm} \hspace{-1.2cm} 
\begin{center}  \begin{tikzpicture} \scriptsize
  \matrix (m) [matrix of math nodes,
    nodes in empty cells,nodes={minimum width=3ex,
    minimum height=3ex,outer sep=2pt},
 column sep=3ex,row sep=3ex]{  
  6    & 0 &  &  & & &  &  \\   
  5    &\bigstar & 0 && &  &  &   \\ 
 4    & 0 & \bigstar & 0 &  &  &   &  \\   
 3    & 0& 0 & \bigstar & 0 & &   & \\  
 2   & 0 & 0 & 0 & \bigstar & 0 &  &  \\            
1     & 0 & 0 & 0 & 0& \bigstar & 0 &  \\        
 0     & \bigstar & 0 & 0 & 0&|[draw=red, circle]|\bigstar& 0& 0  \\       
 \quad\strut &     0  &  1  & 2  & 3 & 4 &5 &6  \\ &&&&&&&& \\}; 
 \draw[thick] (m-1-1.east) -- (m-8-1.east) ;
 \draw[thick] (m-8-1.north) -- (m-8-8.north east) ;
 
  \draw[-stealth, blue]  (m-7-6) -- (m-6-4);
  \draw[-stealth, blue]  (m-7-6) -- (m-5-3);
    \draw[-stealth, blue]  (m-7-6) -- (m-4-2);
  \draw[-stealth, blue]  (m-8-8) -- (m-7-6);

\end{tikzpicture} \vspace{-2em} 
\end{center}
\caption{ The page $E^2_{p,q}$ when $n=6$. There are no nonzero differentials to or from the term $E^r_{n-2,0}$ for $r\geq 2$.}
\label{FigureE2}
\end{figure}

 The spectral sequence (see \autoref{FigureE2}) converges to $H_{p+q}(BA_n(\OK)')$. The only nonzero $E^2$ term on the diagonal $p+q=n-2$ is the term $E^2_{n-2,0} \cong H_{n-2}(\cT_n(\KK))$, and  this term admits no non-zero incoming or outgoing higher differentials. This gives the desired isomorphism.
\end{proof} 

\begin{lemma}[Following  {\cite[Theorem B Step 1]{CP}}]  \label{BykReliso} Let $\OO$ be an integral domain, then there is an isomorphism of $\bZ[\GL_n(\OO)]$-modules
\[\mr{Byk}_n(\OO) \overset{\cong}\lra H_{n-1}(BA_n(\OO),BA_n(\OO)').\]
\end{lemma}
\begin{proof} By construction, $BA_n(\OO)'$ contains all simplices of $BA_n(\OO)$ except for the $n$-simplices corresponding to augmented $n$-frames, and the $(n-1)$-simplices corresponding to non-augmented $n$-frames. From the exact sequence
\[C_n(BA_n(\OO),BA_n(\OO)') \overset{\delta}\longrightarrow C_{n-1}(BA_n(\OO),BA_n(\OO)') \longrightarrow H_{n-1}(BA_n(\OO),BA_n(\OO)') \longrightarrow 0\]
we see that $H_{n-1}(BA_n(\OO),BA_n(\OO)')$ is the group generated by the simplices
\[\{ \{v_1, \ldots, v_{n}\} \; | \;  v_1 \oplus \cdots \oplus v_{n} = \OO^n \}\]
modulo relations of the form 
\begin{align*}
& \delta\big(\{v_0, v_1, v_2, v_3 \ldots, v_{n-1}\}\big)   \qquad  \text{with }\vec{v_0} = \vec{v_1}+\vec{v_2}  \\ 
&= \{v_1, v_2, v_3 \ldots, v_{n-1}\}- \{v_0, v_2, v_3 \ldots, v_{n-1}\} +  \{v_0, v_1, v_3 \ldots, v_{n-1}\} - 0 + 0- \cdots \pm 0.
\end{align*}
This is precisely the presentation defining the group $\mr{Byk}_n(\OO)$. 
\end{proof} 

\begin{lemma}[Following  {\cite[Theorem B Step 2]{CP}}]  \label{LESisos} Let $\OK$ be the Gaussian integers or Eisenstein integers. There is an isomorphism of $\bZ[\GL_n(\OK)]$-modules
\[H_{n-1}(BA_n(\OK),BA_n(\OK)') \overset{\cong}{\longrightarrow} \widetilde H_{n-2}(BA_n(\OK)').\]
\end{lemma} 
\begin{proof} By \autoref{BAnOtherRings}, $BA_n(\OK)$ is $(n-1)$--connected. Thus an isomorphism is given by the connecting homomorphism in the long exact sequence of the pair $(BA_n(\OK),BA_n(\OK)')$. \end{proof} 

Recall that \autoref{BforGandE} says that when $\OK$ is the Gaussian integers or Eisenstein integers, the generalized Bykovski\u\i\, presentation holds.

\begin{proof}[Proof of \autoref{BforGandE}]
Consider the maps \[\mr{Byk}_n(\OK) \xrightarrow[(*)]{\cong}  H_{n-1}(BA_n(\OK),BA_n(\OK)') \xrightarrow[(**)]{\cong}  \widetilde H_{n-2}(BA_n(\OK)') \xrightarrow[(***)]{\cong}  \widetilde H_{n-2}(\cT_n(\KK)),\]
the rightmost group being $\St_n(\KK)$ by definition. The map $(*)$ is an isomorphism by \autoref{BykReliso}, the map $(**)$ is an isomorphism by \autoref{LESisos}, and  the map $(*\! *\! *)$ is an isomorphism by \autoref{MPSSisos}. Thus the composite is an isomorphism of $\bZ[\GL_n(\OK)]$-modules. As in the proof of Church--Putman \cite[Theorem B]{CP}, this composite is the map described in \autoref{SecDualizing}. 
\end{proof}

\begin{corollary} \label{StnRes} Let $\OK$ be the Gaussian integers or Eisenstein integers, and $\KK$ its field of fractions. Then $\St_n(\KK)$ admits a partial resolution by $\Z[\GL_n(\OK)]$-modules
\[C_n(BA_n(\OK),BA_n(\OK)') \overset{\delta}\longrightarrow C_{n-1}(BA_n(\OK),BA_n(\OK)') \longrightarrow 
\St_n(\KK) \longrightarrow 0.\]
\end{corollary}
\begin{proof} By \autoref{BforGandE}, there is an isomorphism $\mr{Byk}_n(\OK) \cong \St_n(\KK)$, and so the result follows from the proof of \autoref{BykReliso}. 
\end{proof} 

We will now use our resolution to show vanishing for group homology with coefficient in the Steinberg module.

\begin{theorem} \label{Sthomology} 
Let $\OK$ denote the  Gaussian integers or Eisenstein integers and let $\bk$ be a ring with $n!$ and $3$ invertible.  Then \begin{align*}
H_1(\GL_n(\OK);\St_n(\KK)\otimes \bk) &= 0 \qquad \text{ for $n \geq 2$}, \\
H_1(\SL_n(\OK);\St_n(\KK)\otimes \bk) &= 0 \qquad \text{ for $n \geq 3$}. 
\end{align*} 
\end{theorem}

\begin{proof} By \autoref{StnRes} we have a partial resolution
\[C_n(BA_n(\OK),BA_n(\OK)';\bk) \overset{\partial}\longrightarrow C_{n-1}(BA_n(\OK),BA_n(\OK)';\bk) \longrightarrow 
\St_n(\KK)\otimes\bk \to 0\]
of $\St_n(\KK)\otimes\bk$ by $\bk[\GL_n(\OK)]$-modules. These modules are flat, e.g.~by Church--Putman \cite[Lemma 3.2]{CP}; its proof only requires that the orders of the stabilizers of simplices in $BA_n(\OK)$ which are not in $BA_n(\OK)'$ are invertible in the coefficients. There are two cases: \begin{itemize}

\item If $\sigma = \{v_1,\ldots,v_n\}$ is a basis of $\OK^n$, it has stabilizer of order $|\OK^\times|^n \cdot n!$, \item If $\sigma = \{v_0,\ldots,v_n\}$ is an additive simplex spanning $\OK^n$, it has stabilizer of order $6 \cdot |\OK^\times|^{n-1} \cdot (n-2)!$. \end{itemize}

\noindent Thus when $G_n$ is $\SL_n(\OK)$ or $\GL_n(\OK)$, we may compute $H_*(G_n; \St_n(\KK)\otimes \bk)$ by extending this partial resolution to a flat resolution, and taking the homology of the chain complex  
\[\cdots \longrightarrow C_n(BA_n(\OK),BA_n(\OK)';\bk)_{G_n} \longrightarrow C_{n-1}(BA_n(\OK),BA_n(\OK)';\bk)_{G_n} \longrightarrow  0\]
obtained by passing to $G_n$-coinvariants. To show that $H_1(G_n; \St_n(\KK)\otimes \bk) =0$, then, it suffices to show that the coinvariants $C_n(BA_n(\OK),BA_n(\OK)';\bk)_{G_n}$ vanish. 

Fix $n \geq 2$. The free $\bk$-module $C_n(BA_n(\OK),BA_n(\OK)';\bk) = C_n(BA_n(\OK);\bk) $ is spanned by augmented frames $\{v_0,\ldots,  v_n\}$ subject to the relation  
\[\{v_0,\ldots,  v_n\}=\mr{sgn}(\sigma)\{v_{\sigma(0)},\ldots, v_{\sigma(n)}\},\]
with $\sigma$ a permutation $\{1,\ldots,n\}$ and $\mr{sgn}(\sigma)$ its sign, and with $G_n$-action
\[g \{ v_0,\ldots,  v_n \}\coloneqq \{g(v_0), \ldots,g(v_n)\}, \qquad g \in G_n.\]
Consider an augmented basis $\{v_0,\ldots,  v_n\}$, reorder, and pick representatives $\vec v_0, \ldots, \vec v_n$ such that $\vec v_0=\vec v_1+\vec v_2$. Let $h \in \GL_n(\OK)$ be the linear map defined by 
\[h(\vec v_1)=\vec v_2, \qquad h(\vec v_2)=\vec v_1, \qquad h(\vec v_i)=\vec v_i  \text{ for $i >2$.}\]
 Then $h(\{v_0,v_1,v_2,  v_3,\ldots,  v_n\})=\{v_0,v_2,v_1, v_3, \ldots,  v_n\}=-\{v_0, v_1, v_2,  v_3,\ldots,  v_n\}$. Since $2$ is invertible in $\bk$, this implies that generators of $C_n(BA_n(\OK);\bk)$ map to zero in $C_n(BA_n(\OK);\bk)_{\GL_n(\OK)}$ and hence \[C_n(BA_n(\OK);\bk)_{\GL_n(\OK)} = 0 \qquad \text{for $n \geq 2$}.\] 
The element $h$ does not have determinant $1$ and this is the reason we must assume $n \geq 3$ to show the $\SL_n(\OK)$-coinvariants vanish. For $n \geq 3$, there is a linear map $\ell \in \SL_n(\OK)$ satisfying
\[\ell(\vec v_1)=\vec v_2, \qquad  \ell(\vec v_2)=\vec v_1,  \qquad \ell(\vec v_3)=-\vec v_3, \qquad \ell(\vec v_i)=\vec v_i  \text{ for $i >3$.}\]
Again $\ell$ negates  the generator $\{v_0,\ldots,  v_n\}$, and we infer that 
 \[C_n(BA_n(\OK);\bk)_{\SL_n(\OK)} \cong 0 \qquad \text{for $n \geq 3$}.\qedhere\] 
 \end{proof}

The following theorem implies \autoref{Vanishing}.
 
\begin{theorem} 
Let $\OK$ denote the  Gaussian integers or Eisenstein integers. Let $\bk$ be a ring with $(n+1)!$ invertible if $n$ is congruent to $1$ modulo $4$ and with $(2n+1)!$ invertible otherwise.  Then \begin{align*}
H^{\nu_n-1}(\GL_n(\OK);\bk) &= 0 \qquad \text{ for $n \geq 2$}, \\
H^{\nu_n-1}(\SL_n(\OK);\bk) &= 0 \qquad \text{ for $n \geq 3$}. 
\end{align*} 
\end{theorem}

\begin{proof} Let $\OK$ denote the Gaussian integers or Eisenstein integers. Borel--Serre duality \cite{BoSe} applies not just with $\Q$-coefficients as stated in the introduction, but in fact their work implies that duality hold with any coefficients in which all torsion primes for the group are invertible. By \cite[Lemma 3.9]{DGGJSY}, the torsion primes of $\GL_n(\OK)$ for $\OK$ quadratic imaginary are bounded by $n+1$ if $n \equiv 1 \pmod{4}$ and $2n+1$ otherwise. Thus, $H_1(\SL_n(\OK);\St_n(\KK)\otimes \bk) \cong H^{\nu_n-1}(\SL_n(\OK);\bk)$ as the orders of torsion elements of $G_n$ are invertible in $\bk$. Similarly, because $\KK$ is quadratic imaginary, Putman--Studenmund \cite[Theorem C and following paragraph]{PutStu} implies that \[H^{\nu_n-1}(\GL_n(\OK);\bk) \overset{\cong}{\lra} H_1(\GL_n(\OK);\St_n(\KK)\otimes \bk).\] 
\autoref{Sthomology} completes the proof.
\end{proof} 

\begin{remark}Church--Putman \cite[Theorem A]{CP} also gave a vanishing result for the twisted cohomology groups $H^{\nu_n-1}(\SL_n(\bZ);V_\lambda)$, where $V_\lambda$ is the rational representation of $\GL_n(\bQ)$ with highest weight $\lambda$ given by the partition $\lambda = (\lambda_1 \geq \ldots \geq \lambda_n)$; it vanishes for $n \geq 3+||\lambda||$ with $||\lambda|| = \sum_{i=1}^n (\lambda_i-\lambda_n)$. Their arguments are easily adapted to our situation: \[H^{\nu_n-1}(\SL_n(\OK);V_\lambda)=H^{\nu_n-1}(\GL_n(\OK);V_\lambda)=0 \qquad \text{for $n \geq 3+||\lambda||$,}\]
where $V_\lambda$ is now the rational representation of $\GL_n(\KK)$ with highest weight $\lambda$.
\end{remark}



\subsection{Examples of the failure of the generalized Bykovski\u\i \, presentation}

In this subsection, we prove \autoref{NoB} which gives examples of rings for which the Bykovski\u{\i} presentation does not hold. This will follow from the following more general theorem.


\begin{theorem} \label{EimpliesNoB}
Let $\OO$ be a Euclidean domain with field of fractions $\KK$. If $\OO$ has detours, then the map $\mr{Byk}_n(\OO) \m \St_n(\KK)$ is not injective for all $n \geq 2$.
\end{theorem}

\begin{proof}
 
 Assume $\OO$ has detours and is Euclidean. 
 
\vspace{.1in}
\noindent {\bf The case $n=2$.} Recall from \autoref{BykReliso} that there is an isomorphism $H_{1}(BA_2(\OO),BA_2(\OO)') \cong \mr{Byk}_2(\OO)$ of $\bZ[\GL_2(\OO)]$-modules. Consider the long exact sequence of the pair $(BA_2(\OO),BA_2(\OO)')$,
\[\begin{tikzcd}
H_1(BA_2(\OO)') \rar & H_1(BA_2(\OO)) \rar \ar[draw=none]{d}[name=X, anchor=center]{} &  H_{1}(BA_2(\OO),BA_2(\OO)') \ar[rounded corners,
to path={ -- ([xshift=2ex]\tikztostart.east)
	|- (X.center) \tikztonodes
	-| ([xshift=-2ex]\tikztotarget.west)
	-- (\tikztotarget)}]{dll} &
\\
H_0(BA_2(\OO)') \rar &  H_0(BA_2(\OO)) \rar & H_{0}(BA_2(\OO),BA_2(\OO)').
\end{tikzcd}\]
\autoref{H1} implies that the group $H_1(BA_2(\OO))$ is nonzero, and as $H_1(BA_2(\OO)') = 0$ the connecting homomorphism $\partial$ has nontrivial kernel. By the proof of \autoref{BforGandE}, the map $\mr{Byk}_2(\OO) \m \St_2(\KK)$ factors as
\[\mr{Byk}_2(\OO) \overset{\cong}{\longrightarrow}  H_{1}(BA_2(\OO),BA_2(\OO)')  \overset{\partial}{\longrightarrow}  \widetilde H_{0}(BA_2(\OO)') \longrightarrow  \widetilde H_{0}(\cT_n(\KK)) \coloneqq\St_2(\KK).\] Since $\partial$ is not injective, the composite $\mr{Byk}_2(\OO) \m \St_2(\KK)$ cannot be an isomorphism. 
 
 \vspace{.1in}
 \noindent {\bf The case $n \geq 3$.} 
Recall from \autoref{BA'} that  $BA_n(\OO)' \subseteq BA_n(\OO)$ is the subcomplex of all simplices $\{v_0, v_1, \ldots, v_p\}$ for which $v_0 + v_1 + \cdots + v_p$ is a proper summand of $\OO^n$. There is a map $$\mr{span} \colon \mr{sd}(BA_n(\OO)') \to \cT_n(\KK)$$ and an associated strongly convergent spectral sequence
\[ E^2_{p,q} = H_p\Big(\cT_n(\KK);  [V \mapsto H_q(\mr{span} \backslash V)]\Big) \Longrightarrow H_{p+q}(BA_n(\OO)'),\]
described in 
\autoref{MPSS}. We will verify that for $n \geq 3$ its $E^2$-page satisfies the following:
\begin{enumerate}[\indent (i)]
	\item For $q=0$, the term $E^2_{p,0}=0$ unless $p=0$ or $p=n-2$.
	\item For $q=1$, the term $E^2_{p,1}=0$ unless $p=n-3$.
	\item For $q \geq 2$, the term $E^2_{p,q}=0$ unless $(p+q)$ is equal to $(n-1)$ or $(n-2)$. \\
	See \autoref{E2noB}.
	\item $E^2_{n-2,0} \cong H_{n-2}(\cT_n(\KK))$.
	\item $\displaystyle E^2_{n-3,1} \cong     \bigoplus_{\substack{ V \subseteq \KK^n \\ \dim(V)=2}} \widetilde{H}_{n-4}(\cT(\KK^n/V); H_1(BA(V))$.
\end{enumerate} 


%
%
%
 \begin{figure}[h!]    \hspace{-0cm} \hspace{-1.2cm} 
\begin{center}  \begin{tikzpicture} \scriptsize
  \matrix (m) [matrix of math nodes,
    nodes in empty cells,nodes={minimum width=3ex,
    minimum height=3ex,outer sep=2pt},
 column sep=3ex,row sep=3ex]{  
  6    & 0 &  &  & & &  &  \\   
  5    &\bigstar & 0 && &  &  &   \\ 
 4    & \bigstar & \bigstar & 0 &  &  &   &  \\   
 3    & 0& \bigstar & \bigstar & 0 & &   & \\  
 2   & 0 & 0 & \bigstar & \bigstar & 0 &  &  \\            
1     & 0 & 0 & 0 & \bigstar& 0 & 0 &  \\        
 0     & \bigstar & 0 & 0 & 0&\bigstar & 0& 0  \\       
 \quad\strut &     0  &  1  & 2  & 3 & 4 &5 &6  \\}; 
 \draw[thick] (m-1-1.east) -- (m-8-1.east) ;
 \draw[thick] (m-8-1.north) -- (m-8-8.north east) ;
 



\end{tikzpicture}
\end{center}
\caption{The page $E^2_{p,q}$ when $n=6$.}
\label{E2noB}
\end{figure}

We first analyze the posets $\mr{span}\backslash V$. For a proper nonzero subspace $V \subseteq \KK^n$, observe that
\[\mr{span}\backslash V  = \Big\{ \{v_0, v_1, \ldots, v_p\} \in \cat{simp}(BA_n(\OO)') \; \Big| \; \KK v_0 + \KK v_1 + \cdots + \KK v_p \subseteq V\Big\} \cong BA(V\cap \OO^n).\]
By abuse of notation, for $V \subseteq \KK^n$ we write $BA(V)$ to denote the complex $BA(V \cap \OO^n)$. \autoref{n2conn} states that the complex $BA(V)$ is $(\dim(V)-2)$-connected, so 
\[H_q(\mr{span}\backslash V ) = 0 \qquad \text{ except possibly when $q=0$, $\dim(V)-1$, or $\dim(V)$}.\]
 
First consider the case $q=0$. The complex $BA(V)$ is always connected; this follows from  \autoref{n2conn} when $\dim(V) \geq 2$ and because $BA(V)$ is a point when $\dim(V)=1$. Thus when $q=0$ we find that the functor $[V \mapsto H_0(BA(V)]$ is the trivial functor $\Z$. Then $E^2_{p,0} \cong H_p(\cT_n(\KK))$. By the Solomon--Tits theorem (\autoref{SolomonTits}), $\cT_n(\KK)$ is spherical of dimension $(n-2)$, so the term $E^2_{p,0}=0$ unless $p=0$ or $p=n-2$, properties (i) and (iv).
 
Let $q=1$. Because $BA(V)$ is a point when $\dim(V)=1$, the group  $H_1(BA(V))$ can be nonzero only when $\dim(V)=2$. Thus the functor $[V \mapsto H_1(\mr{span}\backslash V)]$ is nonzero only on elements of a single height. Taking the quotient by $V$ gives an isomorphism $\cT_n(\KK)_{>V} \m \cT(\KK^n/V)$, and so by \autoref{ChainsSingleSupport}, 
\[E^2_{p,1}  =  \bigoplus_{\substack{ V \subseteq \KK^n \\ \dim(V)=2}} \widetilde{H}_{p-1}(\cT(\KK^n/V); H_1(BA(V)).\]
\noindent The Solomon--Tits theorem implies that $E^2_{p,1}=0$ unless $p=n-3$, properties (ii) and (v).
 
Let $q \geq 2$. In order to apply \autoref{ChainsSingleSupport} to the terms $E^2_{p,q}$, we will write the functors $[V \mapsto H_q(\mr{span}\backslash V ) = H_q(BA(V))]$ as extensions of functors that are each nonzero only on elements of a single height. There is a short exact sequence of functors, 
\[0 \lra H'_q \lra H_q(BA(-)) \lra H''_q \lra 0,\]
with functors $H'_q$, $H''_q$ given by
\[H'_q(V) \coloneqq \begin{cases} H_q(BA(V)) & \text{if $\dim(V) = q+1$,} \\
0 & \text{otherwise,} \end{cases} \qquad H''_q(V) \coloneqq \begin{cases} H_q(BA(V)) & \text{if $\dim(V) = q$,} \\
0 & \text{otherwise,} \end{cases} \]
and natural transformations between them given by
\[\begin{tikzcd} &[10pt] \dim(U) =q &[-3pt] \dim(W) = q+1 \\[-12pt]
H'_q \dar & 0 \rar \dar & H_{q} (BA(W)) \dar[equals] \\[-2pt]
H_q(BA(-)) \dar & H_{q} (BA(U)) \dar[equals] \rar{(U \hookrightarrow W)_*} & H_{q} (BA(W)) \dar \\[-2pt]
H''_q & H_{q} (BA(U)) \rar & 0\end{tikzcd}\]

We can then apply \autoref{ChainsSingleSupport} to the terms in the associated long exact sequence on homology:  
\[\begin{tikzcd} \vdots \dar &[-10pt] \\[-12pt]
H_p(\cT_n(\KK);  H_q') \dar \rar[equal] &\displaystyle  \bigoplus_{\substack{ W \subseteq \KK^n \\ \dim(W)=q+1}} \widetilde{H}_{p-1}(\cT(\KK^n/W); H_q(BA(W))  \\[-5pt]
H_p\Big(\cT_n(\KK),  [V \mapsto H_q(BA(V))]\Big) \rar[equal] \dar & E^2_{p,q} \\[-5pt]
H_p(\cT_n(\KK);  H_q'') \dar \rar[equal] & \displaystyle   \bigoplus_{\substack{ U \subseteq \KK^n \\ \dim(U)=q}} \widetilde{H}_{p-1}(\cT(\KK^n/U); H_q(BA(U)) \\[-26pt]
\vdots & \end{tikzcd}\]

The Solomon--Tits theorem now implies that for $q\geq 2$ the homology groups $E^2_{p,q}$ can be nonzero only when $p+q$ is equal to $n-1$ or $n-2$, property (iii). We have verified our description of the $E^2$ page, as illustrated in \autoref{E2noB}.

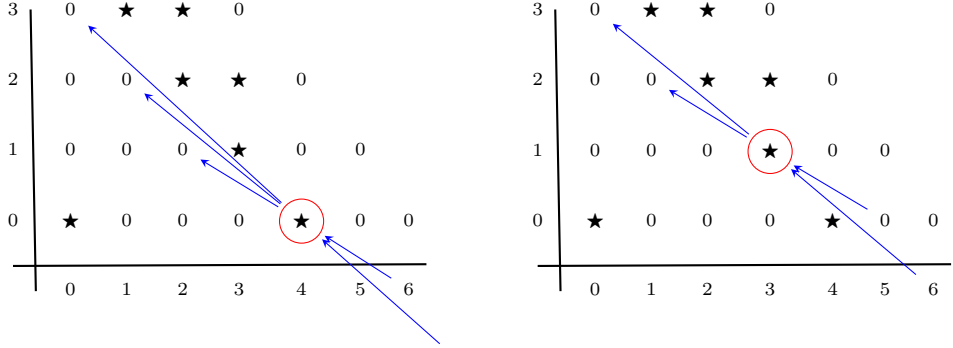
\begin{figure}[h!]    \hspace{-0cm} \hspace{-1.2cm} 
	\begin{subfigure}{.45\textwidth} 
		\begin{tikzpicture} \scriptsize
		\matrix (m) [matrix of math nodes,
		nodes in empty cells,nodes={minimum width=3ex,
			minimum height=5ex,outer sep=2pt},
		column sep=3ex,row sep=3ex]{  
			3    & 0& \bigstar & \bigstar & 0 & &   && \\  
			2   & 0 & 0 & \bigstar & \bigstar & 0 &  &&  \\            
			1     & 0 & 0 & 0 & \bigstar&0 & 0 &&  \\        
			0     & \bigstar & 0 & 0 & 0&|[draw=red, circle]| \bigstar & 0& 0&  \\       
			\quad\strut &     0  &  1  & 2  & 3 & 4 &5 &6 &  \\ &&&&&&&& \\}; 
		\draw[thick] (m-1-1.east) -- (m-5-1.east) ;
		\draw[thick] (m-5-1.north) -- (m-5-8.north east) ;
		
		\draw[-stealth, blue]  (m-4-6) -- (m-3-4);
		\draw[-stealth, blue]  (m-4-6) -- (m-2-3);
		\draw[-stealth, blue]  (m-4-6) -- (m-1-2);
		\draw[-stealth, blue]  (m-5-8) -- (m-4-6);
		\draw[-stealth, blue]  (m-6-9) -- (m-4-6);
		
		\end{tikzpicture}
		\vspace{-2em}
		\subcaption{The term $E^r_{n-2,0}$}
		\label{E^2_{n-2,0}}
	\end{subfigure}
	\begin{subfigure}{.45\textwidth} 
		\begin{tikzpicture} \scriptsize
		\matrix (m) [matrix of math nodes,
		nodes in empty cells,nodes={minimum width=3ex,
			minimum height=5ex,outer sep=2pt},
		column sep=3ex,row sep=3ex]{  
			3    & 0& \bigstar & \bigstar & 0 & &   && \\  
			2   & 0 & 0 & \bigstar & \bigstar & 0 &  &&  \\            
			1     & 0 & 0 & 0 &|[draw=red, circle]|  \bigstar&0 & 0 &&  \\        
			0     & \bigstar & 0 & 0 & 0& \bigstar & 0& 0&  \\       
			\quad\strut &     0  &  1  & 2  & 3 & 4 &5 &6 &  \\ &&&&&&&& \\}; 
		\draw[thick] (m-1-1.east) -- (m-5-1.east) ;
		\draw[thick] (m-5-1.north) -- (m-5-8.north east) ;
		
		\draw[-stealth, blue]  (m-3-5) -- (m-2-3);
		\draw[-stealth, blue]  (m-3-5) -- (m-1-2);
		\draw[-stealth, blue]  (m-4-7) -- (m-3-5);
		\draw[-stealth, blue]  (m-5-8) -- (m-3-5);
		
		\end{tikzpicture}
		\vspace{-2em}
		\subcaption{The term $E^r_{n-3,1}$}
		\label{E^2_{n-3,1}}
	\end{subfigure}
	\caption{The spectral sequence $E^r_{p,q}$ illustrated for $n=6$.}
	\label{FigNoD}
\end{figure}

From the structure of the $E^2$ page, we can deduce the terms
$E^r_{n-2, 0}$ and $E^r_{n-3, 1}$ are not the source or target of any nonzero differentials for any $r \geq 2$. See \autoref{FigNoD}. It follows that $E^2_{n-2,0} \cong E^\infty_{n-2, 0}$ and $E^2_{n-3, 1} \cong E^\infty_{n-3, 1}$. Using formal properties of the spectral sequence, we see that there is a surjection
\[H_{n-2}(BA_n(\OO)') \longrightarrow E^2_{n-2, 0}  \cong H_{n-2}(\cT_n(\KK))\]
and the term $E^2_{n-3, 1}$ is a quotient of the kernel. But 
\[E^2_{n-3, 1}  \cong     \bigoplus_{\substack{ V \subseteq \KK^n \\ \dim(V)=2}} \widetilde{H}_{n-4}(\cT(\KK^n/V); H_1(BA(V)) \cong \bigoplus_{\substack{ V \subseteq \KK^n \\ \dim(V)=2}} \St(\KK^n/V) \otimes H_1(BA(V)).\]
\autoref{H1} imply that the group $H_1(BA(V))$ is nonzero, and the Steinberg module $\St(\KK^n/V)$ is nonzero for $n \geq 3$, so we conclude $E^2_{n-3, 1} \neq 0$ in this range. 

As in the proof of \autoref{BforGandE}, our map $\mr{Byk}_n(\OO) \m \St_n(\KK)$ factors as (since $n\geq 3$, in degree $(n-2)$ we can conflate reduced and non-reduced homology)
\[\mr{Byk}_n(\OO) \overset{\cong}{\longrightarrow}  H_{n-1}(BA_n(\OO),BA_n(\OO)')  \overset{\partial}\lra  \widetilde H_{n-2}(BA_n(\OO)') \lra  \widetilde H_{n-2}(\cT_n(\KK)) = \St_n(\KK).\] Since $ H_{n-2}(BA_n(\OO)) = 0$, the connecting homomorphism $\partial$ is surjective. We have proven that the map $H_{n-2}(BA_n(\OO)') \to H_{n-2}(\cT_n(\KK))$ is not injective, so $\mr{Byk}_n(\OO) \m \St_n(\KK)$ is not injective.
\end{proof}

We now prove \autoref{NoB} which states that the generalized Bykovski\u{\i} presentation does not hold for Euclidean quadratic number rings that are not generated by units.

\begin{proof}[Proof of \autoref{NoB}]

Let $\OK$ be a quadratic number ring which is Euclidean but is not generated by units. By \autoref{exotic}, $\OK$ has detours. The claim now follows from \autoref{EimpliesNoB}.
\end{proof}

\begin{remark}\label{MultUnits} See \autoref{SecDualizing} for our notation for fundamental classes of apartments. In this notation, the proof of \autoref{NoB} in conjunction with \autoref{sqrt7loop} shows that, 
 \[ \left[\left[\begin{pmatrix} 1\\ 0 \end{pmatrix} , \begin{pmatrix} \sqrt{7} \\ 1 \end{pmatrix} \right ] \right ] 
 + \left[\left[ \begin{pmatrix} \sqrt{7} \\ 1 \end{pmatrix}, \begin{pmatrix} 8\\ -3 \end{pmatrix} \right ]\right ]  + \left[ \left[ \begin{pmatrix} 8 \\ -3 \end{pmatrix}, \begin{pmatrix} -3\\ 1 \end{pmatrix} \right ]\right ]   
 + \left[\left[ \begin{pmatrix} -3 \\ 1 \end{pmatrix}, \begin{pmatrix} 1\\ 0 \end{pmatrix} \right ] \right ] =0   \] 
 is a relation in $\St_2(\Q(\sqrt{7}))$ that does not follow from the generalized Bykovski\u\i \, relations.
\end{remark}

\section{Open questions} We end with some open questions. All examples of Euclidean domains for which the generalized Bykovski\u\i \, presentation is known to hold are generated by units. Conversely, all Euclidean domains for which the generalized Bykovski\u\i \, presentation is known to fail are not generated by units.

\begin{question}
For $\OK$ a Euclidean domain, does the generalized Bykovski\u\i \, presentation hold if and only if $\OK$ is generated by units?
\end{question}


The following question asks whether all relations in $\St_n(\KK)$ come from $\St_2(\KK)$.

\begin{question}
Let $\mr{Ker}_n(\OK)$ denote the kernel of $\mr{Byk_n(\OK)} \m \St_n(\KK)$. For $\OK$ a Euclidean domain, is there a natural surjection  \[ \mr{Ind}^{\GL_n(\OK)}_{\GL_{2}(\OK) \times \GL_{n-2}(\OK) } \mr{Ker}_2(\OK) \boxtimes \mr{Byk}_{n-2}(\OK) \m \mr{Ker}_n(\OK)?\]
\end{question}

The group $\mr{Ker}_n(\OK)$ measures relations in the Steinberg module beyond those appearing in the Bykovski\u\i \, presentation. An affirmative answer would imply that for $n \geq 4$ and all Euclidean domains $\OK$, we have $H_1(\GL_n(\OK);\St_n(\KK) \otimes \Q) = H_1(\SL_n(\OK);\St_n(\KK) \otimes \Q) = 0$.

Vanishing results near the virtual cohomological dimension are the subject of several conjectures by Church--Farb--Putman \cite{CFPconj}. In particular, they conjecture ( \cite[Conjecture 2]{CFP}) that 
\[H^{\nu_n-i}(\SL_n(\bZ);\bQ) = 0 \qquad \text{ for $i<n-1$}.\]
This is supported by the available computations \cite[Remark 5.3]{DSEVKM}, and known for $i=0$ by Lee--Szczarba \cite[Theorem 1.3]{LS} and $i=1$ by Church--Putman \cite[Theorem A]{CP}. It is natural to ask the same question for other number rings; using \cite[Theorem 1.3]{LS} and \autoref{Vanishing} it is also true for $i=0,1$ when we replace $\bZ$ by the Gaussian integers or Eisenstein integers.

\begin{conjecture} Let $\OK$ denote the Gaussian integers or Eisenstein integers. Then \[H^{\nu_n-i}(\SL_n(\OK);\bQ) = H^{\nu_n-i}(\GL_n(\OK);\bQ) = 0 \qquad \text{ for all $i<n-1$}.\]
\end{conjecture}

This is supported by the available computations Dutour Sikiri\'{c}--Gangl--Gunnells--Hanke--Sch\"{u}rmann--Yasaki 
\cite[Tables 11, 12]{DGGJSY} or \cite[Propositions 2.6, 2.10]{DSGGHSYK4}:  
$$H^{\nu_4-i}(\GL_4(\OK);\bQ) = 0 \qquad \text{ for $i \leq 2$ and $\OK$ the Gaussian integers or Eisenstein integers.}$$  For general Euclidean number rings, one might expect a similar vanishing result though possibly with a worse range.

One can also ask about integral versions of our vanishing result. In \cite[Theorem 1.10]{MNP} it was proven that for $n \geq 6$ \[H_1(\GL_n(\Z);\St_n(\Q)) = H_1(\SL_n(\Z);\St_n(\Q)) = 0.\] 

\begin{question}
Is it true that for $n \geq 6$ we have
 \[H_1(\GL_n(\OK);\St_n(\KK)) = H_1(\SL_n(\OK);\St_n(\KK)) = 0\]  when $\OK$ is the Gaussian integers or Eisenstein integers? Can the range be improved?
\end{question}

The Bykovski\u\i\, presentation is also useful for computing the homology of congruence subgroups. For an ideal $J \subset \OK$, let \[ \Gamma_n(J)\coloneqq \mr{Ker}\Big[\GL_n(\OK) \to \GL_n(\OK/J) \Big].  \] We say an ideal $J$ has \emph{the Lee--Szczarba property} (see \cite[Page 28]{LS}) if the natural map \[H^{\nu_n}(\Gamma_n(J);\Q) \lra H_0(\Gamma_n(J);\St_n(\KK) \otimes \Q) \lra \widetilde H_{n-2}(\cT_n(\KK)/\Gamma_n(J);\Q)\] is an isomorphism.  For $\OK=\Z$, the prime ideals with the Lee--Szczarba property are $(2),(3)$ and $(5)$; see \cite[Theorem 1.2]{LS}  and \cite[Theorem A]{MPP}. The proof relies on the Bykovski\u\i\, presentation.

\begin{question}
Which prime ideals in the Gaussian integers or Eisenstein integers have the Lee--Szczarba property? 
\end{question}

Often, $\widetilde H_{n-2}(\cT_n(\KK)/\Gamma_n(J))$ is computable (see \cite[Table 1]{MPP}) so an answer to this question could yield concrete calculations.

\bibliographystyle{amsalpha}
\bibliography{refs}

\vspace{.5cm}

\end{document}